\documentclass[11pt,reqno,oneside]{amsart}
\usepackage{amsmath,amsthm,amsfonts,amssymb,mathrsfs,extarrows,mathtools}
\usepackage{MnSymbol}
\usepackage{dsfont}

\usepackage[top=1in,bottom=1in,a4paper]{geometry}
\usepackage{xcolor}
\definecolor{monokaiPink}{HTML}{F92672}
\definecolor{monokaiGreen}{HTML}{A6E22E}
\definecolor{monokaiOrange}{HTML}{FD971F}
\definecolor{monokaiCyan}{HTML}{66D9EF}
\definecolor{monokaiPurple}{HTML}{AE81FF}
\definecolor{monokaiYellow}{HTML}{E6DB74}
\usepackage{hyperref}
\hypersetup{colorlinks=true,
	linktoc=page,
	linkcolor=cyan,
	citecolor=cyan,
	urlcolor=monokaiCyan,
	hidelinks
}
\usepackage{xurl}

\usepackage[shortlabels]{enumitem}
\setlist[enumerate]{leftmargin=*}
\setlist[itemize]{label={--}}
\usepackage[nameinlink]{cleveref}
\usepackage{tikz-cd}
\tikzcdset{arrow style=math font}
\usepackage[indent]{parskip}

\title[Monodromy Rank and Semisimple MTC for HK varieties]{Monodromy Rank and the Semisimple Mumford--Tate Conjecture for Hyper-Kähler Varieties}

\date{Aug 2026}
\author{Zhichao Tang}
\address{Shanghai Center for Mathematical Science \\ Fudan University \\ Shanghai \\ 200438 \\ China}
\email{\href{mailto:zctang19@fudan.edu.cn}{zctang19@fudan.edu.cn}}

\author{Haitao Zou}
\address{Universität Bielefeld \\ Universitätsstraße 25 \\ 33615 Bielefeld \\Germany}
\email{\href{mailto:hzou@math.uni-bielefeld.de}{hzou@math.uni-bielefeld.de}}

\newcommand{\ZZ}{\mathds{Z}}
\newcommand{\QQ}{\mathds{Q}}
\newcommand{\RR}{\mathds{R}}
\newcommand{\CC}{\mathds{C}}
\newcommand{\FF}{\mathds{F}}
\newcommand{\GG}{\mathbb{G}}

\newcommand{\VV}{\mathbb{V}}
\newcommand{\DTorus}{\mathbb{S}}

\newcommand{\bC}{\mathbf{C}}
\newcommand{\bG}{\mathbf{G}}

\newcommand{\bL}{\mathbf{L}}
\newcommand{\bM}{\mathbf{M}}
\newcommand{\bP}{\mathbf{P}}
\newcommand{\bR}{\mathbf{R}}

\newcommand{\cM}{\mathcal{M}}

\newcommand{\cO}{\mathcal{O}}

\newcommand{\fg}{\mathfrak{g}}
\newcommand{\fh}{\mathfrak{h}}

\newcommand{\fX}{\mathfrak{X}}

\newcommand{\sG}{\mathscr{G}}

\newcommand{\locc}{{\itshape loc.~cit.}}

\DeclareMathOperator{\rH}{H}

\DeclareMathOperator{\SH}{SH}
\DeclareMathOperator{\LLV}{LLV}
\DeclareMathOperator{\Image}{Im}
\DeclareMathOperator{\prim}{pr}
\DeclareMathOperator{\Sym}{Sym}
\DeclareMathOperator{\MT}{\mathbf{MT}}
\DeclareMathOperator{\HS}{\mathbf{HS}}
\DeclareMathOperator{\Vect}{\mathbf{Vect}}
\DeclareMathOperator{\mt}{\mathfrak{mt}}

\DeclareMathOperator{\alg}{alg}
\DeclareMathOperator{\Ad}{Ad}
\DeclareMathOperator{\GL}{GL}

\DeclareMathOperator{\SO}{SO}
\DeclareMathOperator{\OO}{O}

\DeclareMathOperator{\Spin}{Spin}
\DeclareMathOperator{\GSpin}{GSpin}

\DeclareMathOperator{\Lie}{Lie}
\DeclareMathOperator{\Aut}{Aut}
\DeclareMathOperator{\End}{End}
\DeclareMathOperator{\Spec}{Spec}
\DeclareMathOperator{\Gal}{Gal}

\DeclareMathOperator{\red}{red}
\DeclareMathOperator{\rk}{rk}
\DeclareMathOperator{\semisimple}{ss}
\DeclareMathOperator{\tw}{tw}
\DeclareMathOperator{\der}{der}

\DeclareMathOperator{\cha}{char}
\DeclareMathOperator{\Res}{Res}

\DeclareMathOperator{\Tr}{Tr}

\DeclareMathOperator{\fsl}{\mathfrak{sl}}
\DeclareMathOperator{\gl}{\mathfrak{gl}}
\DeclareMathOperator{\fso}{\mathfrak{so}}

\DeclareMathOperator{\id}{\mathds{1}}
\DeclareMathOperator{\mot}{mot}
\DeclareMathOperator{\AM}{\mathbf{Mot}}

\DeclareMathOperator{\et}{\Acute{e}t}
\DeclareMathOperator{\dR}{dR}
\DeclareMathOperator{\an}{an}

\DeclareMathOperator{\Hdg}{Hdg}

\DeclareMathOperator{\pst}{pst}
\DeclareMathOperator{\conn}{conn}

\DeclareMathOperator{\Hom}{Hom}
\DeclareMathOperator{\MTC}{MTC}
\DeclareMathOperator{\SMTC}{SMTC}
\DeclareMathOperator{\sAb}{\mathscr{A}}
\DeclareMathOperator{\Kthree}{K3}
\DeclareMathOperator{\Kum}{Kum}
\DeclareMathOperator{\OG}{OG}

\DeclareMathOperator{\cl}{cl}
\DeclareMathOperator{\CH}{CH}
\DeclareMathOperator{\Frob}{Frob}

\DeclareMathOperator{\Zar}{Zar}
\DeclareMathOperator{\gr}{gr}
\DeclareMathOperator{\Gr}{\mathbf{Gr}}

\DeclareMathOperator{\Exc}{Exc}
\DeclareMathOperator{\specialize}{sp}
\DeclareMathOperator{\VHS}{VHS}
\DeclareMathOperator{\simplyconnected}{sc}

\DeclarePairedDelimiterX\Set[1]\lbrace\rbrace{\def\given{\;\delimsize\vert\;}#1}
\newcommand{\isomto}{\stackrel{\textstyle\sim}{\smash{\longrightarrow}\rule{0pt}{0.4ex}}}
\newtheorem{theorem}[subsubsection]{Theorem}
\newtheorem{mainthm}{Theorem}

\newtheorem{prop}[subsubsection]{Proposition}
\newtheorem{cor}[subsubsection]{Corollary}
\newtheorem{lemma}[subsubsection]{Lemma}

\theoremstyle{definition}
\newtheorem{definition}[subsubsection]{Definition}
\newtheorem{notation}[subsubsection]{Notation}
\newtheorem{example}[subsubsection]{Example}
\newtheorem{remark}[subsubsection]{Remark}
\newtheorem{conj}[subsubsection]{Conjecture}

\newtheorem*{MainConj}{Conjecture}
\numberwithin{equation}{subsection}

\AddToHook{env/theorem/begin}{\crefalias{subsubsection}{theorem}}
\AddToHook{env/prop/begin}{\crefalias{subsubsection}{prop}}
\AddToHook{env/cor/begin}{\crefalias{subsubsection}{cor}}
\AddToHook{env/lemma/begin}{\crefalias{subsubsection}{lemma}}
\AddToHook{env/definition/begin}{\crefalias{subsubsection}{definition}}
\AddToHook{env/notation/begin}{\crefalias{subsubsection}{notation}}
\AddToHook{env/example/begin}{\crefalias{subsubsection}{example}}
\AddToHook{env/remark/begin}{\crefalias{subsubsection}{remark}}
\AddToHook{env/conj/begin}{\crefalias{subsubsection}{conj}}
\AddToHook{env/hypothesis/begin}{\crefalias{subsubsection}{hypothesis}}

\crefname{mainthm}{Theorem}{Theorems}
\Crefname{mainthm}{Theorem}{Theorems}
\crefname{prop}{Proposition}{Propositions}
\Crefname{prop}{Proposition}{Propositions}
\crefname{cor}{Corollary}{Corollaries}
\Crefname{cor}{Corollary}{Corollaries}
\crefname{lemma}{Lemma}{Lemmas}
\Crefname{lemma}{Lemma}{Lemmas}
\crefname{notation}{Notation}{Notations}
\Crefname{notation}{Notation}{Notations}
\crefname{conj}{Conjecture}{Conjectures}
\Crefname{conj}{Conjecture}{Conjectures}
\crefname{hypothesis}{Hypothesis}{Hypotheses}
\Crefname{hypothesis}{Hypothesis}{Hypotheses}

\begin{document}

\begin{abstract}
We study the Mumford--Tate conjecture for hyper-K\"ahler varieties. We identify the Mumford--Tate group with a Levi factor of the connected total $\ell$-adic monodromy group. It follows that the Mumford--Tate conjecture holds after semisimplification in every cohomological degree. We call this the semisimple Mumford--Tate conjecture. As applications, we derive a Hodge-to-Tate implication for powers, prove deformation invariance of the Mumford--Tate conjecture, establish the $\ell$-adic Nagai conjecture for Type~I reduction, and extend Hui--Larsen's hyperspecial maximality theorem from degree two to total cohomology. The proof combines Pink's generation theorem for weak Hodge cocharacters with a multiplicity-weighted direct-sum construction and a rigidity argument for the graded cohomology algebra.
\end{abstract}



\maketitle

\tableofcontents

\section{Introduction}
Mumford and Tate formulated the following conjecture in their study of Galois representations attached to abelian varieties \cite{Mum66}; see also \cite{Ser64}. Let $X$ be a smooth projective variety over a finitely generated extension $K/\QQ$, and fix an embedding $K\hookrightarrow\CC$.
\begin{MainConj}[$\MTC_i$]\label{conj:Mumford--Tate conjecture}
Let $\overline K$ be the algebraic closure of $K$ in $\CC$. For every prime $\ell$, the Artin comparison isomorphism identifies the ambient general linear groups and, under this identification, one has
		\[
		\bG_{\ell,i}(X)^{\circ} =\MT_{i}(X_{\CC})_{\QQ_{\ell}}\,.
		\]
		Here $\bG_{\ell,i}(X)$ is the \emph{$\ell$-adic algebraic monodromy group} of $\rH^i_{\et}(X_{\overline K},\QQ_\ell)$, and $\MT_i(X_\CC)$ is the \emph{Mumford--Tate group} of the Hodge structure on $\rH^i_B(X_\CC,\QQ)$.
\end{MainConj}

We develop an approach to the Mumford--Tate conjecture in arbitrary cohomological degree for \emph{hyper-Kähler varieties}; see \Cref{def:HK}. We prove that the conjectural equality holds after semisimplification and that $(\MTC_i)$ is invariant under deformation.

\subsection{The abelian case and two difficulties}
We first recall the case of abelian varieties. Let $A/K$ be an abelian variety. Deligne proved that every Hodge cycle on $A$ is absolute Hodge \cite[I,~2.11]{DMOS82}. It follows that there is a natural inclusion
\begin{equation}\label{eq:Hodge vs absolute Hodge for abelain varieties}
\bG_{\ell,i}(A)^{\circ} \hookrightarrow \MT_{i}(A_{\CC})_{\QQ_{\ell}}
\end{equation}
for every $0\leq i\leq2\dim A$. Given this inclusion, the Mumford--Tate conjecture is equivalent to equality of dimensions. It remains open in general, even for abelian fourfolds; see \cite{MZ95,Pink98}.

The same argument applies more generally to \emph{abelian motives}, a class containing, for example, the motives of curves, Fermat hypersurfaces, and K3 surfaces \cite[II, Proposition~6.26]{DMOS82}. Here ``motive'' may mean a Chow motive or a motive in the sense of Deligne \cite[II,~\S6]{DMOS82} or André \cite{Andre96b}.

The motivic formulation also extends to Shimura varieties, viewed as moduli spaces of Hodge structures of abelian motives with additional structure. Deligne's theory of canonical models supplies the relevant Galois representations, and hence a Mumford--Tate conjecture at generic points of special subvarieties; see, for example, \cite{UY13,Vasiu08}.

Two difficulties arise beyond the abelian setting. First, the absolute Hodge conjecture is not known for a general smooth projective variety, so the analogue of \eqref{eq:Hodge vs absolute Hodge for abelain varieties} is unavailable. An abstract isomorphism $\bG_{\ell,i}(X)^\circ\cong\MT_i(X_\CC)_{\QQ_\ell}$ would not by itself identify these groups inside
\[
\GL\bigl(\rH^i_B(X_\CC,\QQ)\otimes_\QQ\QQ_\ell\bigr)
\]
under comparison.

Second, the expected semisimplicity of geometric Galois representations is open in general. The Mumford--Tate conjecture implies that $\bG_{\ell,i}(X)^\circ$ is reductive because Mumford--Tate groups are reductive. For abelian varieties over finitely generated extensions of $\QQ$, this reductivity follows from Faltings's theorem on endomorphisms; see \cite[IV,~\S1, 1.1 and 1.4]{RationalPoints}.

Reductivity of $\bG_{\ell,i}(X)^\circ$ is equivalent to semisimplicity of $\rH^i_{\et}(X_{\overline K},\QQ_\ell)$ as a $G_K$-representation. It should not be confused with Frobenius semisimplicity: Chebotarev density does not promote semisimplicity of almost all Frobenius operators to semisimplicity of the global representation. We refer to \cite{Tate65,Moonen19} for the related forms of the Tate conjecture. We call reductivity of $\bG_\ell(X)^\circ$ the \emph{semisimplicity conjecture} for $X$; it is known in only a few cases.

\subsection{Hyper-Kähler varieties and the semisimple Mumford--Tate conjecture}
We now specialize to hyper-Kähler varieties.
\begin{definition}\label{def:HK}
		A \emph{hyper-Kähler variety} over a field $K$ of characteristic zero is a smooth projective geometrically connected variety $X/K$ such that
	\[
	\pi_1^{\et}(X_{\overline K})=1
	\qquad\text{and}\qquad
	\rH^0(X,\Omega^2_{X/K})=K\sigma
	\]
		for a nondegenerate $2$-form $\sigma\colon\cO_X\to\Omega^2_{X/K}$.
\end{definition}

The symplectic form forces $\dim X=2n$. In dimension two, the hyper-Kähler varieties are precisely the smooth projective K3 surfaces.

In higher dimensions, Beauville \cite{Bea83} constructed two series of deformation types:
\begin{itemize}
    \item varieties of $\Kthree^{[n]}$-type, and
    \item varieties of $\Kum_n$-type.
\end{itemize}
O'Grady constructed two further deformation types in dimensions $6$ and $10$ \cite{OG6,OG10}, denoted by $\OG6$ and $\OG10$. These four series comprise all deformation types currently known.

In degree two, the Mumford--Tate conjecture for $b_2\geq4$ is due to Tankeev \cite{Tankeev90,Tankeev95} in dimension two and to André \cite{Andre96} in higher dimension. The only remaining value is $b_2=3$: the $(2,0)$- and $(0,2)$-summands and an ample $(1,1)$-class are linearly independent.

\begin{theorem}
		\label{thm:Mumford--Tate conjecture HK degree 2}
		Let $X$ be a hyper-Kähler variety over a finitely generated extension $K/\QQ$. Then $(\MTC_2)$ holds for $X$.
\end{theorem}
The proof, including the case $b_2(X)=3$, is given in \Cref{subsec:degree two MTC}. We ask whether $(\MTC_i)$ in higher degree can be reduced to $(\MTC_2)$. Our method studies the degree-two projection of the total $\ell$-adic algebraic monodromy group $\bG_\ell(X)$.

Let $\MT(X_\CC)$ be the Mumford--Tate group of the full cohomology ring $\rH_B^\bullet(X_\CC,\QQ)$, viewed as a graded polarizable Hodge structure. Verbitsky's theorem (\Cref{thm:MT in LLV}) implies that the projection
		\[
		\MT(X_{\CC}) \longrightarrow \MT_2(X_{\CC}),
		\]
is an isogeny of degree at most $2$; see \Cref{lem:project MT group}. We prove the following rank analogue for $\ell$-adic algebraic monodromy groups.
\begin{mainthm}[\Cref{thm:rank of ell algebraic monodromy groups}]\label{mainthm:ranks are all equal}
Let $X$ be a hyper-Kähler variety over a finitely generated extension $K/\QQ$. For every prime $\ell$,
\[
 \rk\bG_{\ell}(X)
 =\rk\bG_{\ell,+}(X)
 =\rk \bG_{\ell,2}(X)\,.
\]
\end{mainthm}
Here $\bG_\ell(X)$ is the $\ell$-adic algebraic monodromy group of $\rH^\bullet_{\et}(X_{\overline K},\QQ_\ell)$, and $\rk G=\rk G^\circ$ denotes the dimension of a maximal torus of the identity component of its geometric form.

Together with the Levi identification below, the rank theorem compares reductive monodromy and Mumford--Tate groups without assuming that the motive $\fh(X)$ is abelian.

Choose a degree-preserving semisimplification of the total cohomology representation, and write
\(
\bG_{\ell}^{\mathrm{ssm}}(X)
\)
for its connected algebraic monodromy group. It is well defined up to degree-preserving conjugacy. A compatible total Levi subgroup $\bL_\ell$ realizes this semisimplification, and we set $\bG_{\ell,i}^{\red}(X)=\pi_{\ell,i}(\bL_\ell)$ in degree $i$.

The rank theorem (\Cref{mainthm:ranks are all equal}) shows that the connected kernel of the degree-two projection is unipotent. Combining this with $(\MTC_2)$ and the specialization argument of \S\ref{sec:Mumford--Tate conjecture for HK varieties}, we obtain an isomorphism
	\begin{equation}\label{eq:MTC in Lie algebras}
	\fg_{\ell}^{\red}(X) \cong \mt(X_{\CC})_{\QQ_{\ell}}
	\end{equation}
		of $\QQ_\ell$-Lie algebras. This is the Lie-algebra form of Mumford's original conjecture for abelian varieties \cite{Mum66}. We prove the stronger group-level statement under the comparison isomorphism.
\begin{mainthm}[\Cref{thm:semisimple Mumford--Tate conjecture even part}]
\label{mainthm:semisimple Mumford-Tate conjecture}
Let $X$ be a hyper-Kähler variety over a finitely generated extension $K/\QQ$. For every prime $\ell$, there is a compatible total Levi subgroup $\bL_\ell\subseteq\bG_\ell(X)^\circ$. On setting $\bG_{\ell,i}^{\red}(X)=\pi_{\ell,i}(\bL_\ell)$, comparison gives an identification
\[
\bG_{\ell,i}^{\red}(X)= \MT_{i}(X_{\CC})_{\QQ_{\ell}}
\]
for every embedding $K\hookrightarrow\CC$ and every $0\leq i\leq2\dim X$. In fact, the compatible total Levi is unique, and its images give the unique degreewise Levi subgroups in the chosen motivic realization.
\end{mainthm}

We call this equality the \emph{semisimple Mumford--Tate conjecture} in degree $i$. If the degree-$i$ Galois representation is semisimple, then $\bG_{\ell,i}^{\red}(X)=\bG_{\ell,i}(X)^\circ$, so the usual Mumford--Tate conjecture follows.

The semisimple statement also relates Hodge and Tate classes. In particular, \Cref{mainthm:semisimple Mumford-Tate conjecture} gives
\begin{equation}\label{eq:MT in monodromy}
\MT(X_{\CC})_{\QQ_{\ell}} \subseteq \bG_{\ell}(X)^{\circ}
\end{equation}
for every prime $\ell$. Thus every $G_K$-invariant tensor is Mumford--Tate invariant. Under comparison, it therefore lies in the $\QQ_\ell$-span of rational Hodge tensors.

Together with the Künneth formula, this yields the following implication.
\begin{cor}\label{cor:Hodge implies Tate}
		Let $X$ be a hyper-Kähler variety over a finitely generated extension $K/\QQ$. If the Hodge conjecture holds in codimension $k$ for $X_\CC^{\times m}$, then, for every prime $\ell$, the cycle-class map
	\[
	\CH^k(X^{\times m})\otimes_\QQ\QQ_\ell \longrightarrow \rH^{2k}_{\et}(X_{\overline K}^{\times m},\QQ_\ell(k))^{G_K}
	\]
	is surjective.
\end{cor}
\begin{proof}
By \eqref{eq:MT in monodromy}, comparison carries every $G_K$-invariant class into the $\QQ_\ell$-span of rational Hodge classes. The standard Hodge-to-Tate argument now applies; see \cite[Proposition~2.3.2 and Remark~2.2.6(ii)]{Moonen17b}.
\end{proof}

\subsection{Mumford--Tate conjecture in a family}
Floccari--Fu--Zhang \cite{FFZ21} and Soldatenkov \cite{Sol22} used the deformation principle for motivated cycles to prove that abelianity of André motives attached to hyper-Kähler varieties is deformation invariant. Their result implies the corresponding invariance of the motivic Mumford--Tate conjecture.

We instead work directly with the $\ell$-adic local systems in a family. We prove that semisimplicity is invariant under deformation and then use \Cref{mainthm:semisimple Mumford-Tate conjecture} to obtain the same conclusion for $(\MTC_i)$ in every degree.

Recall that two varieties $X/K$ and $Y/K'$ are \emph{deformation equivalent} if there is a smooth projective family $f \colon \fX \to S$ over a geometrically connected smooth variety $S$ (defined over a common finitely generated subfield $K_0 \subseteq K \cap K'$), such that $X \cong \fX_s$ and $Y \cong \fX_{s'}$ for some $s \in S(K)$ and $s' \in S(K')$.

\begin{mainthm}[\Cref{cor:MTC in families}]\label{mainthm:Semisimplicity in family}
    Let $X$ be a hyper-Kähler variety over a finitely generated extension $K/\QQ$. For every $0\leq i\leq2\dim X$, the following statements are equivalent.
    \begin{enumerate}[(a)]
        \item $(\MTC_{i})$ holds for $X$.
        \item For every prime $\ell$, the $G_K$-representation $\rH^{i}_{\et}(X_{\overline K},\QQ_\ell)$ is semisimple.
        \item For every hyper-Kähler variety $Y/K'$ over a finitely generated extension $K'/\QQ$ that is deformation equivalent to $X$, $(\MTC_i)$ holds for $Y$.
    \end{enumerate}
\end{mainthm}
\begin{proof}
\begin{itemize}
    \item The equivalence (a) $\Longleftrightarrow$ (b) is \Cref{lem:semisimple Mumford--Tate conjecture}.
    \item For (a) $\implies$ (c), choose a family realizing the deformation equivalence. The two fibers are finite-type fibers, so \Cref{cor:MTC in families} applies.
    \item The implication (c) $\implies$ (a) follows by taking $Y=X$. \qedhere
\end{itemize}
\end{proof}

The following corollary was proved in \cite{FFZ21,Sol22}; \Cref{mainthm:Semisimplicity in family} gives an alternative proof.
\begin{cor}\label{cor:MumfordTate for known types}
If $X$ has deformation type $\Kthree^{[n]}$, $\Kum_n$, $\OG6$, or $\OG10$, then $(\MTC_i)$ holds for every prime $\ell$ and every $0\leq i\leq2\dim X$.
\end{cor}

\subsection{Applications}

Although semisimplification discards extension data, the semisimple Mumford--Tate conjecture suffices for several arithmetic applications.
\subsubsection{Arithmetic Nagai conjecture}
The Nagai conjecture asks to what extent monodromy in higher degree is determined by monodromy in degree two. We use its $\ell$-adic formulation in terms of local monodromy operators \cite{IIKTZ25}.

In \S\ref{subsec:Nagai}, we prove the $\ell$-adic Nagai conjecture for hyper-Kähler varieties with Type I reduction over a $p$-adic local field. The result was previously known only for the four known deformation types \cite{IIKTZ25}.

\begin{mainthm}[\Cref{thm:Nagai for Type I}]\label{mainthm:Nagai for Type I}
	Let $X$ be a hyper-Kähler variety over a $p$-adic local field $K_v$. Suppose $X$ has Type I reduction, i.e. $\rH^2_{\et}(X_{\overline{K}_v},\QQ_{\ell'})$ is potentially unramified for some prime $\ell'\neq p$. Then $\rH^{i}_{\et}(X_{\overline{K}_v},\QQ_{\ell})$ is potentially unramified for all $0\leq i\leq4n$ and all primes $\ell\neq p$.
\end{mainthm}
The proof places a Levi factor containing the semisimple part of Frobenius inside the global Mumford--Tate group and uses the fact that this Levi factor centralizes the defect group.

\subsubsection{Maximality of Galois action (after Hui--Larsen)}\label{sss:Maximality}
Serre \cite{Ser94} asked whether compatible systems attached to maximal motives have maximal Galois image.

Larsen \cite[\S0]{Lar95} formulated the following version for compatible systems arising from arbitrary smooth projective varieties. Let $\bG_{\ell}^{\circ} \to \bG_\ell^{\red}$ be the reductive quotient of $\bG_\ell^\circ$ and put
\[
 \bG_\ell^{\mathrm{ss}}=\bG_\ell^{\red}/Z(\bG_\ell^{\red})^\circ.
\]
Let $\bG_\ell^{\simplyconnected}\to\bG_\ell^{\mathrm{ss}}$ be its simply connected covering. If $\Gamma_\ell=\rho_\ell(G_K)$, define $\Gamma_\ell^{\simplyconnected}$, as in Hui--Larsen, to be the inverse image in $\bG_\ell^{\simplyconnected}(\QQ_\ell)$ of the image of $\Gamma_\ell\cap\bG_\ell^\circ(\QQ_\ell)$ in $\bG_\ell^{\mathrm{ss}}(\QQ_\ell)$.

\begin{conj}[Larsen]
		\label{conj:maximality of Galois action}
		Let
		\[
		\{\rho_\ell\colon G_K\to
		\GL(\rH_{\et}(X_{\overline K},\QQ_\ell))\}_{\ell\in\mathcal P}
		\]
		be a $\QQ$-compatible system arising from a smooth projective variety $X$ over a finitely generated extension $K/\QQ$. For all sufficiently large $\ell$, the group $\Gamma_\ell^{\simplyconnected}$ is a maximal compact subgroup of $\bG_\ell^{\simplyconnected}(\QQ_\ell)$. Moreover, it is \emph{hyperspecial}: there is a smooth affine group scheme $\mathscr G_\ell^{\simplyconnected}$ over $\ZZ_\ell$ such that
		\[
		\mathscr G_\ell^{\simplyconnected}(\ZZ_\ell)
		=\Gamma_\ell^{\simplyconnected}.
		\]
\end{conj}

Hui--Larsen proved Larsen's conjecture in degree two for hyper-Kähler varieties \cite[Theorem~1.3]{HL20}. The rank theorem extends their result to the full cohomology.

\begin{mainthm}[\Cref{thm:maximality of Galois action}]\label{mainthm:maximality of Galois action}
			Let $X$ be a hyper-Kähler variety over a finitely generated extension $K/\QQ$. Then $\Gamma_{\ell}^{\simplyconnected}$ is a hyperspecial maximal compact subgroup of $\bG_{\ell}(X)^{\simplyconnected}(\QQ_{\ell})$ for all sufficiently large primes $\ell$.
\end{mainthm}

\subsection{Strategy of proof}
The proof of our main results has three steps.

\subsubsection{From Pink generators to rank comparison}
Over a number field, a Levi factor realizes the semisimplification of the total monodromy group. Pink's theorem (\Cref{thm:Pink weak Hodge generation}) generates this group by weak Hodge cocharacters. Since total cohomology mixes the cohomological degrees, we apply Pink's theorem to a multiplicity-weighted direct sum. Uniqueness of base-$M$ expansion then recovers the Hodge multiplicities degree by degree.

Let $\bR$ be the image of the rational twisted LLV representation. Its projection to degree two has finite kernel, while a graded-algebra automorphism acting trivially in degree two centralizes $\bR$. Once $(\MTC_2)$ places the degree-two projections in $\bR_2$, a trace-rigidity argument forces the weak Hodge cocharacters into $\bR$. The finite degree-two projection then identifies every Levi factor over a number field with the Mumford--Tate group. Specialization proves \Cref{mainthm:ranks are all equal} over every finitely generated extension of $\QQ$.

The decisive observation is that this rigidity is intrinsic to the graded cohomology algebra (\Cref{prop:defect rigidity weak Hodge}) and remains valid when $b_2=3$.

\subsubsection{From rank comparison to a canonical Levi factor}
For $b_2>3$, we place $X$ in a polarized family with maximal degree-two monodromy. At a Galois-generic fiber, degree two determines the derived subgroup of a Levi factor; the weight torus determines its center. Specialization and deformation invariance of the twisted LLV representation then place every Levi factor in $\bR$. Since $\bR\to\bR_2$ has finite kernel, $(\MTC_2)$ identifies that Levi factor with the Mumford--Tate group. For $b_2=3$, the descent of \Cref{lem:b2=3 descends to number field} reduces directly to the number-field case. Thus the Levi factor is canonical, not merely abstractly isomorphic to the Mumford--Tate group.

\subsubsection{Semisimplicity and deformation}
The canonical Levi factor identifies the reductive quotient in every degree with the corresponding Mumford--Tate group. Hence $(\MTC_i)$ is equivalent to semisimplicity of the degree-$i$ Galois representation. Cadoret's specialization results \cite{Cad15}, together with the behavior of the unipotent radical under cospecialization, show that this semisimplicity is constant in a connected family. This gives \Cref{mainthm:semisimple Mumford-Tate conjecture,mainthm:Semisimplicity in family}.

\subsection*{Outline}
Section~\ref{sec:Tannakian formalism and motivic Galois group} recalls Andr\'e motives and compatible systems, records the reduction to number fields used throughout the paper, and proves the degree-two Mumford--Tate conjecture (\Cref{thm:Mumford--Tate conjecture HK degree 2}). Section~\ref{sec:LLV representations of hyper-Kahler varieties} develops the LLV and twisted LLV representations used in the main arguments.

Section~\ref{sec:Rank estimation} proves the rank theorem (\Cref{mainthm:ranks are all equal}) and identifies Levi factors over number fields. Section~\ref{sec:Mumford--Tate conjecture for HK varieties} constructs the canonical Levi factor, proves \Cref{mainthm:semisimple Mumford-Tate conjecture}, and establishes deformation invariance in \Cref{mainthm:Semisimplicity in family}. It also recovers the Mumford--Tate conjecture for the four known deformation types.

Section~\ref{sec:applications} proves the arithmetic applications: the arithmetic Nagai conjecture for Type I reduction and the maximality theorem for the Galois action.
\subsection*{Notation}
\begin{itemize}
    \item A variety over $K$ is an integral separated scheme of finite type over $K$.
    \item We write $X^{\cl}$ for the set of closed points of $X$. This set is often denoted by $|X|$; we reserve $|-|$ for cardinality.
    \item If an object $Y$ is defined over $E$ and $F/E$ is a field extension, then $Y_F$ denotes its base change to $F$.
    \item For an algebraic group $\bG$, we write $\bG^{\circ}$ for its identity component and $\bG^{\der}$ for its derived subgroup.
\end{itemize}
\subsection*{Acknowledgments}
We are grateful to Salvatore Floccari, Lie Fu, and Kazuhiro Ito for helpful discussions. We also thank Zhiyuan Li and Ziquan Yang for useful suggestions.

Z.~Tang is supported by the NSFC grant (No.~12121001). H.~Zou is supported by the Deutsche Forschungsgemeinschaft (DFG, German Research Foundation), Project-ID 491392403, TRR 358.
\section{Motivic Galois groups and compatible systems}\label{sec:Tannakian formalism and motivic Galois group}
In this section, we summarize the elementary facts concerning André motives and the algebraic groups arising from them. We also fix the notation that will be frequently used in the subsequent discussion.
\subsection{André motives and realizations}
\label{sec:motivated cycle}
The theory of motivated cycles, developed by Y. André \cite{Andre96b}, gives rise to the category of André motives and offers a practical framework for analyzing algebraic groups arising from various cohomological realizations.
\subsubsection{}
Fix a field embedding $K \subseteq \CC$. Let $X$ be a smooth projective variety over $K$. Recall that a class $\alpha \in \rH^{2i}(X_{\CC},\QQ)$ is \emph{motivated} if there is a smooth projective variety $Y$ over $K$ and algebraic classes $\beta$, $\gamma \in \rH^{\bullet}((X\times_K Y)_{\CC},\QQ)$ such that $\alpha = p_*(\gamma \cup \star \beta)$ (with $p \colon X \times_K Y \to X$ the projection, and $\star$ the Hodge star operator). The category $\AM_K$ of André motives consists of objects $(X,\pi,m)$ where
\begin{itemize}
	\item $X$ is a smooth projective variety over $K$,
	\item $\pi$ is a motivated cycle on $X \times_KX$ such that $\pi \circ \pi = \pi$, and
	\item $m$ is an integer.
\end{itemize}
Writing $A^r_{\mot}(X\times_KY)$ for the motivated cycles of codimension $r$, the morphisms in $\AM_K$ are
\[
 \Hom_{\AM_K}\bigl((X,\pi_X,m),(Y,\pi_Y,n)\bigr)
 =\pi_Y\circ A^{\dim X-m+n}_{\mot}(X\times_KY)\circ\pi_X.
\]
Denote by $\fh(X)$ the André motive $(X,\Delta_X,0)$ of a smooth projective variety over $K$. The category $\AM_K$ satisfies the following properties (see \cite[Theorem 0.4]{Andre96b}):
\begin{itemize}[--]
	\item It has a natural tensor product structure and is a graded Tannakian category over $\QQ$;
	\item It is semisimple and polarized.
\end{itemize}
The grading of $\AM_K$ induces the Chow--Künneth decomposition for each smooth projective variety $X$ over $K$:
\[
\fh(X) \cong \bigoplus_{i=0}^{2\dim X} \fh^i(X)\,,
\]
where $\fh^i(X) = (X, \Delta_i,0)$, since the Künneth factor $\Delta_i \in \rH^{2\dim X}((X \times_K X)_{\CC},\QQ)$ is motivated (see \cite[Proposition 2.2.]{Andre96b}).

\subsubsection{} There is a \emph{Betti realization functor}
\[
\rH_B \colon \AM_K \longrightarrow \Vect_{\QQ}\,,
\]
such that $\rH_B(\fh(X)) = \rH^{\bullet}(X_{\CC},\QQ)$ is the graded $\QQ$-algebra of Betti cohomology for any smooth projective variety $X$ over $K$.

On the other hand, one may also consider the $\ell$-adic realization
\[
\rH_{\ell} \colon \AM_K \longrightarrow \Vect_{\QQ_{\ell}}\,,
\]
such that $\rH_{\ell}(\fh(X)) = \rH^{\bullet}_{\et}(X_{\overline{K}},\QQ_{\ell})$ is the $\ell$-adic étale cohomology, a graded $\QQ_{\ell}$-algebra endowed with a natural $G_K$-action, where $G_K$ is the absolute Galois group of $K$. The Artin comparison provides a natural equivalence from the $\ell$-adic realization $\rH_{\ell}$ to the composition of the following functors
\[
\rH_{B,\QQ_{\ell}} \colon \AM_K \xlongrightarrow{\rH_B} \Vect_{\QQ} \xlongrightarrow{- \otimes_{\QQ} \QQ_{\ell}} \Vect_{\QQ_{\ell}}
\]
for any prime $\ell$.
\subsection{Motivic Galois groups and Mumford--Tate conjecture}
As before, we fix a subfield $K \subseteq \CC$. We denote by $\bG_{\mot}$ the motivic Galois group of $\AM_K$, i.e., the automorphism group of the fiber functor $\rH_B$, which is a reductive pro-algebraic group over $\QQ$. For a single smooth projective variety $X/K$, we can similarly define its motivic Galois group as follows.
\begin{definition}\label{def:motivic Galois group}
Let $X$ be a smooth projective variety over $K$. The \emph{motivic Galois group} $\bG_{\mot}(X)$ of $X$ is the automorphism group of the restriction, as a fiber functor, of the Betti realization functor $\rH_B$ to the Tannakian subcategory generated by the Andr\'e motive $\fh(X)$ and its dual $\fh(X)^{\vee}$ in $\AM_{K}$.
\end{definition}
\subsubsection{}\label{sssec:Grading of motive}
For any $\cM\in\AM_K$, let $\bG_{\mot}(\cM)$ denote its motivic Galois group. Its identity component is reductive because $\AM_K$ is semisimple. This group depends on the chosen embedding $K\hookrightarrow\CC$, but only up to an inner twist; see \cite[Remarks on p.~25]{Andre96b}.
\begin{notation}\label{notation:group with degree}
For a smooth projective variety $X/K$, set
\[
\fh^\bullet(X)=\fh(X),\qquad
\fh^+(X)=\bigoplus_{k=0}^{\dim X}\fh^{2k}(X),\qquad
\fh^-(X)=\bigoplus_{k=0}^{\dim X-1}\fh^{2k+1}(X).
\]
For $\square\in\{\bullet,+,-,0,\ldots,2\dim X\}$, the subscript $\square$ indicates the corresponding total, even, odd, or degreewise realization. For example, $\bG_{\mot,\square}(X)$ is the motivic Galois group of $\fh^\square(X)$. We usually omit $\bullet$. We write $\pi_i$ for the degree-$i$ projection on motivic or Mumford--Tate groups and $\pi_{\ell,i}$ for the corresponding $\ell$-adic projection.
\end{notation}
\subsubsection{}
As mentioned above, $\AM_K$ has a grading defined by the Betti realization functor; namely, $\rH_B$ factors as
\[
\AM_K \longrightarrow \Gr\Vect_{\QQ} \xlongrightarrow{\text{for.}} \Vect_{\QQ}\,.
\]
The Tannakian group of the forgetful functor of the category of graded $\QQ$-vector spaces is just $\GG_m$, the multiplicative group over $\QQ$. The Tannakian formalism induces a homomorphism
 \[\GG_m \longrightarrow \bG_{\mot}\,,\]
denoted by $\omega$. For any smooth projective variety $X$ over $K$, the action of $\GG_m$ on $\fh^{i}(X) \subseteq \fh(X)$ is given by $ z \mapsto z^i$, coinciding with the cohomology degree of $X$.
\subsubsection{}\label{sec:Mumford--Tate group}
Let $\phi \colon \DTorus \to \GL_{\RR}(V_{\RR})$ be a pure Hodge structure on a $\QQ$-vector space $V$, i.e., an algebraic homomorphism from the Deligne torus $\DTorus \coloneqq \Res_{\CC/\RR} \GG_{m,\CC}$. The \emph{Mumford--Tate group $\MT(V,\phi)$} associated to $(V,\phi)$ is the smallest $\QQ$-algebraic subgroup of $\GL(V)$ such that $\phi(\DTorus) \subseteq \MT(V,\phi)(\RR)$. From the definition, it is clear that $\MT(V,\phi)$ is \emph{connected}. When $(V,\phi)$ is polarizable, the Mumford--Tate group $\MT(V,\phi)$ is, moreover, \emph{reductive}. Denote by $\HS_{\QQ}$ the Tannakian category generated by polarizable $\QQ$-Hodge structures.

Let $\cM\in \AM_K$. The Betti realization $\rH_B(\cM)$ of $\cM$ is naturally a polarizable $\QQ$-Hodge structure. For a smooth projective variety $X$ over $K$, we denote by $\MT(X_{\CC})$ the Mumford--Tate group of the graded $\QQ$-Hodge structure on $\rH_B(\fh(X)) = \rH^{\bullet}(X_{\CC},\QQ)$.

It is also convenient to describe Mumford--Tate groups via the Tannakian formalism. The forgetful functor $\text{for.} \colon \HS_{\QQ} \to \Vect_{\QQ}$ is a fiber functor on $\HS_{\QQ}$. Then the Mumford--Tate group of a pure $\QQ$-Hodge structure $V$ is isomorphic to
\[
	\Aut^{\otimes}(\text{for.}|_{\langle V \rangle})\,,
\]
the automorphism group of the restricted fiber functor on the sub-Tannakian category $\langle V \rangle$ generated by $V$ (see \cite[Lemma 2 \& Remark]{Andre92}). In general, for any object $V \in \HS_{\QQ}$, i.e., a subquotient of a direct sum $\oplus_{m} V_m$ in which $V_m$ is a polarizable pure $\QQ$-Hodge structure of weight $m$, we can define its Mumford--Tate group in the same way. In summary, via the Tannakian formalism,
\begin{enumerate}[label={(\arabic*)},leftmargin=*]
		\item We can view the Mumford--Tate group as the \emph{maximal} subgroup of $\GL_{\QQ}(V)$ whose induced action on $V^{\otimes}$ fixes all Hodge tensors of type $(0,0)$.
	\item The Betti realization factors as
	\[
	\begin{tikzcd}[row sep =small, column sep = small]
		\langle \fh(X) \rangle^{\otimes} \ar[rd] \ar[rr,"\rH_B"] & &\Vect_{\QQ} \mathrlap{\,.} \\
		& \HS_{\QQ} \ar[ru,"\text{for.}"']
	\end{tikzcd}
	\]
\end{enumerate}
Also note that the Betti realization $\rH_B$ (with respect to the fixed field embedding $K \subseteq \CC$) induces an injective homomorphism
	\begin{equation}\label{eq:MT in motivic Galois group}
		\MT(X_{\CC}) \hookrightarrow \bG_{\mot}(X)^{\circ}\,,
	\end{equation}
since the motivated cycles are Hodge tensors. Here $\MT(X_{\CC})$ is the Mumford--Tate group of the (graded) polarizable $\QQ$-Hodge structure on $\rH^{\bullet}(X_{\CC},\QQ)$.
\begin{itemize}
	\item The Tannakian category $\HS_{\QQ}$ has a natural $\ZZ$-grading; in particular, there is a homomorphism
	\[
	\omega_{\phi} \colon \GG_{m,\RR} \longrightarrow \MT(V,\phi)_{\RR}
	\]
	This homomorphism is defined over $\QQ$ for any Hodge structure $(V,\phi)$ and is called the \emph{weight cocharacter} of $(V,\phi)$.
	\item If $V = \rH_B(\fh(X))$, then the composition of $\omega_{\phi}$ with the injective homomorphism $\MT(X_{\CC}) \hookrightarrow \bG_{\mot}(X)$ is the homomorphism $\omega \colon \GG_m \to \bG_{\mot}(X)$ in \Cref{sssec:Grading of motive}. Clearly, $\omega_{\phi}$ is nontrivial if and only if $V$ has nonzero weights, and its image lies in the center: $\omega_{\phi}(\GG_{m}) \subseteq Z(\MT(V,\phi))$.
\end{itemize}

\subsubsection{}
Let $K$ be a field of characteristic zero and let $G_K \coloneqq \Gal(\overline{K}/K)$ be its absolute Galois group. For simplicity, we assume that there is a field embedding $K \hookrightarrow \CC$, and we fix one.
\begin{definition}
Let $\rho_{\ell} \colon G_K \to \GL_{\QQ_{\ell}}(V)$ be a continuous finite-dimensional representation of $G_K$. The \emph{$\ell$-adic algebraic monodromy group} of $V$ is the Zariski closure of the image $\Image(\rho_{\ell}) \subseteq \GL_{\QQ_{\ell}}(V)$, denoted by $\bG_{\ell}(V)$.

Via the Tannakian formalism, we can identify the group $\bG_{\ell}(V)$ as the Tannakian group of the restriction of the forgetful functor from the category of $G_K$-representations to the category of $\QQ_{\ell}$-vector spaces $\Vect_{\QQ_{\ell}}$.
\end{definition}

\subsubsection{}
For any $\cM \in \AM_K$, the Artin comparison induces an isomorphism of $\QQ_{\ell}$-algebraic groups
\[
\Aut^{\otimes}\left( \rH_{\ell}|_{\langle\cM \rangle^{\otimes }} \right) \cong \Aut^{\otimes}\left( \rH_{B,\QQ_{\ell}}|_{\langle \cM \rangle^{\otimes}}\right) = \bG_{\mot}(\cM)_{\QQ_{\ell}}
\]
for any prime $\ell$. Therefore, via the Tannakian duality, there is an injective homomorphism
\begin{equation}\label{eq:algebraic monodromy group in MT}
\bG_{\ell}(\cM)^{\circ} \hookrightarrow \bG_{\mot}(\cM)_{\QQ_{\ell}}\,.
\end{equation}
Thus the Mumford--Tate conjecture holds for $\cM$ if and only if \[\bG_{\ell}(\cM)^{\circ} = \MT(\cM_{\CC})_{\QQ_{\ell}}\] as $\QQ_{\ell}$-subgroups of $\bG_{\mot}(\cM)_{\QQ_{\ell}}$.
\subsubsection{}
Keep the convention of \Cref{notation:group with degree}. For a smooth projective variety $X/K$, put
\[
V_{\ell,i}(X)\coloneqq \rH^i_{\et}(X_{\overline K},\QQ_\ell),\quad V_{\ell,\bullet}(X)\coloneqq \bigoplus_iV_{\ell,i}(X),
\]
and let $V_{\ell,+}(X)$ (resp.~$V_{\ell,-}(X)$) be the even-degree (resp.~odd-degree) part of $V_{\ell,\bullet}$. We suppress $(X)$ from these symbols when $X$ is fixed and write
\[
\bG_{\ell,\square}(X)\coloneqq
\bG_\ell\bigl(V_{\ell,\square}(X)\bigr),
\qquad
\bG_{\ell,\bullet}(X)=\bG_\ell(X).
\]
For a chosen Levi subgroup, we use $\bL_{\ell,\square}$; once uniqueness is established, $\bG_{\ell,\square}^{\red}(X)$ denotes the canonical Levi. The corresponding Lie algebras are denoted by
\[
\fg_\ell(X)=\Lie\bG_\ell(X),
\qquad
\fg_{\ell,i}(X)=\Lie\bG_{\ell,i}(X).
\]

We will use the following general consequence of the purity of Frobenius eigenvalues.
\begin{lemma}\label{lem:homotheties in algebraic monodromy group}
Let $X/K$ be a smooth projective variety over a finitely generated field $K/\QQ$. Then the base change of the motivic weight cocharacter factors through the connected algebraic monodromy group, and its image is central:
\begin{equation}\label{eq:homotheties in algebraic monodromy group}
\omega(\GG_m)_{\QQ_\ell}
\subseteq Z\bigl(\bG_\ell(X)^\circ\bigr)
\subseteq\bG_{\mot}(X)_{\QQ_\ell}.
\end{equation}
\end{lemma}
\begin{proof}
Assume first that $K$ is a number field. After a finite extension, the total monodromy group is connected. Choose a place $v$ of good reduction with residue cardinality $q_v$ and residue characteristic different from $\ell$. Let $t_v$ be the semisimple part of a Frobenius lift, replacing it by a positive power so that the Zariski closure $T_v$ of $\langle t_v\rangle$ is connected. Thus $T_v\subseteq\bG_\ell(X)^\circ$ is a torus.

Over $\overline{\QQ}_\ell$, write $\chi_{d,j}$ for the characters of $T_v$ occurring on $\rH^d_{\et}(X_{\overline K},\QQ_\ell)$. If $\prod_{d,j}\chi_{d,j}^{n_{d,j}}=1$, evaluation at $t_v$ and the Weil bounds give
\[
 1=\left|\prod_{d,j}\chi_{d,j}(t_v)^{n_{d,j}}\right|
   =q_v^{a\sum_{d,j}n_{d,j}d/2},
\]
where $a\geq1$ is the power used to define $t_v$. Hence $\sum_{d,j}n_{d,j}d=0$. Because the total representation of $T_v$ is faithful, the characters $\chi_{d,j}$ generate $X^*(T_v)$. The assignment $\chi_{d,j}\mapsto d$ therefore defines a homomorphism $X^*(T_v)\to\ZZ$, and hence a cocharacter
\[
 \omega_\ell\colon\GG_{m,\overline{\QQ}_\ell}
 \longrightarrow T_{v,\overline{\QQ}_\ell}
\]
that acts as $z^d$ on $\rH^d_{\et}(X_{\overline K},\QQ_\ell)$. This action is defined over $\QQ_\ell$, so faithfulness descends $\omega_\ell$ to $\QQ_\ell$; it is the base change of the motivic weight cocharacter. Since the total monodromy group preserves cohomological degree, this scalar degreewise action commutes with it and is therefore central. The required Weil bounds and Frobenius-torus construction are recalled in \cite[Theorems 3.2--3.3 and the discussion preceding (3.4)]{Pink98}.

For general finitely generated $K/\QQ$, use the spreading-out $\fX\to B$ from \Cref{reduction to number field} and choose any closed point $s$ after shrinking $B$ so that the fiber is smooth and projective. Along an étale path, smooth proper base change identifies the cohomology of $\fX_s$ with that of $X$, and the specialized Galois image is contained in the generic algebraic monodromy group. The number-field argument puts the common degreewise weight cocharacter in $\bG_\ell(\fX_s)^\circ$, hence in $\bG_\ell(X)^\circ$. Its scalar degreewise action again makes it central.
\end{proof}

\subsection{Compatibility of systems of Galois representations}\label{subsec:algebraic monodromys and compatible systems}

In this section, we focus on the case where $K$ is a \emph{number field} and on \emph{$\QQ$-compatibility} for a system of $G_K$-representations.
\subsubsection{}
\label{definition of compatibility over number field}
Consider a profinite group $\mathscr{G}$ and a system of continuous $n$-dimensional representations
\[
\rho_{\bullet} = \Set*{\rho_{\ell}\colon \mathscr{G} \longrightarrow \mathrm{GL}_n(\QQ_{\ell})}_{\ell \in \mathcal P}\,,
\]
indexed by a set $\mathcal P$ of rational primes.
Suppose that $\mathscr{G}$ is endowed with a \emph{dense} subset of ``Frobenius elements'', \[\{F_{v}\}_{v\in \Sigma} \subset \mathscr{G}\,.\]
For example, if $\mathscr{G} = G_K$ is the absolute Galois group of a number field $K$ and $F_{v} = \Frob_v$ is a Frobenius representative at a place $v$ of $K$, then the Frobenius conjugacy classes over all finite places form a dense subset by the Chebotarev density theorem.

The system $\rho_{\bullet}$ is called a \emph{$\QQ$-compatible system} of $\ell$-adic representations if there is a subset $\mathscr X \subset \Sigma \times \mathcal P$ such that
\begin{enumerate}
	\item For every $v\in \Sigma$, $(v,\ell)\in \mathscr{X}$ for all but at most finitely many $\ell \in \mathcal P$.
	\item For any primes $\ell_1,\dots,\ell_m\in\mathcal P$, the set $\{F_{v}\;|\; (v,\ell_i)\in\mathscr{X}\;\text{for all }i=1,\dots,m\}$ is dense in $\mathscr{G}$.
	\item For every $(v,\ell)\in\mathscr X$, the characteristic polynomial of $\rho_{\ell}(F_{v})$ lies in $\QQ[t]$ and is independent of $\ell$.
\end{enumerate}
\begin{example}\label{compatibiliy over number fields}
In the geometric context, we consider the following situation. Let $K$ be a number field.
\begin{itemize}[leftmargin=*]
    \item Let $\cM = \fh^{\square}(X)$ be a motive associated to a smooth proper variety $X$ over $K$.
    \item $\Sigma$ is the following subset of the places of $K$:
\begin{equation}\label{eq:good reduction place}
	\Set*{\parbox{17em}{non-archimedean places $v$ of $K$ that are unramified over $\QQ$ and at which $X$ has good reduction.}}\,.
\end{equation}
    \item Let $F_v=\Frob_v\in G_K$ be any lift of an arithmetic Frobenius element. Its conjugacy class modulo inertia is canonical, and its image under every representation unramified at $v$ is therefore well defined up to conjugacy.
\end{itemize}
Since $X$ has good reduction at $v$, the system of Galois representations
	\[
	\rho_{\bullet} = \Set*{\rho_{\ell} \colon G_K \longrightarrow \GL_{\QQ_{\ell}}\left(\rH_{\ell}(\cM)\right)}_{\ell\in \mathcal P}
	\]
	is unramified at all $v \in \Sigma$ by smooth proper base change when $\ell \neq \cha(k_v)$. The system $\rho_{\bullet}$ is $\QQ$-compatible because condition (3) follows from \cite{KatzMessing}.
\end{example}

\begin{remark}
    In general, it is not clear whether the system of $\ell$-adic realizations of an André motive $\cM$ is $\QQ$-compatible or not, except when $\cM \subseteq \fh^{\square}(X)$ is a submotive cut out by an algebraic cycle.
\end{remark}

\subsubsection{} Let $\rho_{\bullet}$ be a $\QQ$-compatible system of representations of $G_K$.
For simplicity, denote by $\bG_{\ell}$ the $\ell$-adic algebraic monodromy group of $\rho_{\ell}(G_K)$ for each prime $\ell \in \mathcal P$. As before, $\bG_{\ell}^{\circ}$ is the connected component of the identity. The set of connected components $\bG_{\ell}/\bG_{\ell}^{\circ}$ is finite, so there is a finite extension $K^{\conn}$ of $K$, corresponding to the finite-index subgroup $\sG=\rho_{\ell}^{-1}(\bG^{\circ}_{\ell}(\QQ_{\ell})) \subseteq G_K$, such that
\[\rho_{\ell}(G_{K^{\conn}}) \subset \bG_{\ell}^{\circ}(\QQ_{\ell})\] is Zariski dense. A priori, it is unclear whether there exists a single finite extension that works uniformly for all $\ell \in \mathcal P$ when $\mathcal P$ is infinite. Nevertheless, we have the following theorem of Serre (see \cite{RibetLetter}, or alternatively Larsen--Pink \cite[Proposition 6.14]{LarsenPink92}).
\begin{theorem}\label{thm:independence of connected components}
		The open subgroup of finite index
		\[\sG\subset G_K\] is independent of $\ell$.
		In particular, the groups $\bG_{\ell}/\bG^{\circ}_{\ell}$ for different $\ell$ are canonically isomorphic; if $\bG_{\ell}$ is connected for some $\ell$, then it is so for all $\ell$.
\end{theorem}
\begin{remark}
Therefore, given a $\QQ$-compatible system of representations $\rho_{\bullet}$ (even when $\mathcal P$ is infinite), we can always assume that $K^{\conn}=K$ after a uniform finite extension by \Cref{thm:independence of connected components}.
\end{remark}
\begin{remark}
    Let $\bG_{\ell}$ be the $\ell$-adic algebraic monodromy group of the motive $\fh(X)$. For any submotive $\cM \subseteq \fh(X)$, e.g., $\fh^i(X)$, there is a surjection $\bG_{\ell} \twoheadrightarrow \bG_{\ell}(\cM)$. Therefore, if $K = K^{\conn}$, then $\bG_{\ell}(\cM)$ is also connected.
\end{remark}

\subsubsection{}\label{reduction to number field}
We record explicitly the spreading-out argument used below. Let $K$ be a finitely generated field over $\QQ$, and let $k_0$ be the algebraic closure of $\QQ$ in $K$; thus $k_0$ is a number field. After localizing a finitely generated $k_0$-subalgebra of $K$, we obtain
\[
 B=\Spec A,\qquad B/k_0\text{ is smooth and geometrically integral},
 \qquad k(B)=K.
\]
After shrinking $B$, a smooth projective variety $X/K$, together with a chosen polarization, extends to a smooth projective morphism $f\colon\fX\to B$. Every closed point $s\in B$ has residue field $k(s)$ finite over $\QQ$, hence is a number-field point.

For a fixed finite set of primes, the Frattini--Hilbert irreducibility argument for $\ell$-adic local systems gives infinitely many closed points $s\in B$ for which, after identifying the geometric cohomology of the generic and special fibers along an étale path, the corresponding algebraic monodromy groups agree. In particular, for any fixed $\ell$ we may choose $s$ such that
\[
 \bG_{\ell}(\fX_s)^{\circ}=\bG_{\ell}(X)^{\circ}.
\]
See \cite[\S 3.1.1]{Cad15}. Consequently, an assertion about a fixed connected $\ell$-adic monodromy group that is invariant under this identification may be checked on a suitable number-field fiber. This is the only form of ``reduction to a number field'' used below.

\subsection{The degree-two Mumford--Tate conjecture}
\label{subsec:degree two MTC}
We now prove \Cref{thm:Mumford--Tate conjecture HK degree 2}. The only case not covered by the results of Tankeev and André is $b_2=3$. We first record the descent to a number field used in that case.

\begin{lemma}
\label{lem:b2=3 descends to number field}
Let $X$ be a hyper-Kähler variety over a finitely generated field $K/\QQ$ with $b_2(X)=3$. After a finite extension $K'/K$, there are a number field $E\subseteq K'$ and a hyper-Kähler variety $X_0/E$ such that $E$ is algebraically closed in $K'$ and
\[
 X_{K'}\cong X_0\times_E K'.
\]
For compatible embeddings into $\CC$, the Mumford--Tate groups of $X$ and $X_0$ are identified. The $\ell$-adic representations of $X_{K'}$ and $X_0$ have the same image; consequently, the connected algebraic monodromy groups of $X$ and $X_0$ are identified.
\end{lemma}
\begin{proof}
After a finite extension, choose a polarization of $X$. By local Torelli, the deformation space of the polarized pair has dimension $b_2(X)-3=0$. Its point on the finite-type Deligne--Mumford moduli stack of polarized hyper-Kähler varieties therefore has residual gerbe over a number field. After a finite extension of that number field the gerbe is neutral; after a further finite extension of $K$, the residue gerbe can be trivialized. This gives a model $X_0/E$ and the required isomorphism. This is the rank-three polarized-descent argument; compare the moduli construction in \cite[Theorem~4.5.2]{Bindt21}.

Enlarge $E$ to its relative algebraic closure in $K'$, which is still a number field. Then $K'/E$ is regular, so the restriction map $G_{K'}\to G_E$ is surjective. The pulled-back $\ell$-adic representations consequently have the same image. The assertion for Mumford--Tate groups follows from the displayed base-change isomorphism and a compatible complex embedding.
\end{proof}

\begin{proof}[Proof of \Cref{thm:Mumford--Tate conjecture HK degree 2}]
The case $b_2(X)\geq4$ is due to Tankeev \cite{Tankeev90,Tankeev95} when $\dim X=2$ and to André \cite{Andre96} in higher dimension. Suppose that $b_2(X)=3$. We first assume that $K$ is a number field. After a finite extension, choose an ample class $h\in\rH^2_{\et}(X_{\overline K},\QQ_\ell(1))^{G_K}$. Since $h^{2,0}(X)=h^{0,2}(X)=1$, one has $h^{1,1}(X)=1$. On the weight-zero twist put
\[
 \rH_B=\rH^2_B(X,\QQ)(1)=\QQ h\oplus T_B,
 \qquad \dim_\QQ T_B=2,
\]
where $T_B$ is the orthogonal complement of $h$ for the Lefschetz form
\[
 \phi_h(x,y)=\int_Xx\cup y\cup h^{\dim X-2}.
\]
The Hodge types of $T_B$ are $(1,-1)$ and $(-1,1)$. The class $h$ and the Lefschetz form $\phi_h$ are rational Hodge tensors, so $\MT(\rH_B)\subseteq\SO(T_B,\phi_h)$. The Hodge cocharacter is nontrivial; hence $\MT(\rH_B)$ is a nontrivial connected subgroup of this one-dimensional torus, and therefore
\[
 \MT(\rH_B)=\SO(T_B,\phi_h).
\]

The same Lefschetz form is Galois invariant on $\rH_\ell=\rH^2_{\et}(X_{\overline K},\QQ_\ell)(1)$: the twists of the two inputs, $h^{\dim X-2}$, and the trace map add up to the top-degree twist. Consequently, writing $\bG_\ell(\rH_\ell)$ for the algebraic monodromy group of this representation,
\[
 \bG_\ell(\rH_\ell)^\circ\subseteq
 \SO(T_\ell,\phi_{h,\ell}),
 \qquad T_\ell=(\QQ_\ell h)^\perp.
\]
This connected group is nontrivial. Indeed, otherwise its action on $T_\ell$ would become trivial after a finite extension, whereas smooth proper $p$-adic Hodge theory gives the two Hodge--Tate weights $-1$ and $1$ on $T_\ell$; see \cite{Fal88}. Thus it is the full one-dimensional torus. Comparison identifies it with $\MT(\rH_B)_{\QQ_\ell}$. The Mumford--Tate conjecture is invariant under Tate twist by \cite[Remark~1.8(ii)]{Moonen17a}, so $(\MTC_2)$ follows. This is the rank-two argument used in the proof of \cite[Corollary~9.3]{Moonen17a}.

For a general finitely generated $K/\QQ$, apply \Cref{lem:b2=3 descends to number field}. The number-field argument and the equality of Galois images prove the twisted statement over $K$. Untwisting completes the proof.
\end{proof}

\section{LLV representations of hyper-Kähler varieties}
\label{sec:LLV representations of hyper-Kahler varieties}

\subsection{Looijenga--Lunts--Verbitsky Lie algebra}
We now briefly review the basic properties of the \emph{Looijenga--Lunts--Verbitsky Lie algebra} (LLV algebra) introduced by Verbitsky \cite{Ver96} and Looijenga--Lunts \cite{LL97}.
\subsubsection{}
Let $K \subseteq \CC$ be a subfield, and let $X$ be a smooth projective variety over $K$ as before. Let $\rH(X) \coloneqq \rH^{\bullet}_{\square}(X)$ denote a cohomological realization of $\fh(X)$, with $\square\in\{B,\ell,\dR,\pst\}$. Let $E$ be the coefficient field of $\rH(-)$.

The LLV algebra $\fg$ for $X/K$ and the associated LLV decomposition of $\rH(X)$ generalize the usual Hard Lefschetz $\fsl_2$-decomposition of the cohomology for an ample class $\omega$ on $X$. Recall that $\omega$ defines two operators on cohomology: the Lefschetz operator $L_\omega=\omega\cup -$, given by cup product, and the dual Lefschetz operator $\Lambda_\omega = \star^{-1} L_{\omega} \star$, where $\star$ is the Hodge star operator. The operators $L_\omega$ and $\Lambda_\omega$ generate an $\fsl_2\subset \mathfrak{gl}(\rH(X))$ acting on $\rH(X)$, and the Hard Lefschetz theorem is equivalent to the existence of the $\fsl_2$-decomposition of cohomology. A more formal framework that avoids the use of the Hodge star operator was formulated in \cite{LL97}, where the key observation is the following identity
\[[L_\omega,\Lambda_\omega]=h\,,\]
with $h$ the shifted degree operator
\begin{equation} \label{eq_def_h}
	h \colon \rH (X) \longrightarrow \rH (X),\;\; x \longmapsto (k - \dim X) x , \;\; \text{for $x \in \rH^k (X)$}\,.
\end{equation}
It is then clear that $\{L_\omega, h,\Lambda_\omega\}$ forms an $\fsl_2$-triple. Moreover, the operator $L_x = x \cup -$ is well defined for any cohomology class $x \in \rH^2(X)$, while $h$ is independent of the choice of $x$. By the Jacobson--Morozov theorem, the existence (and thus the uniqueness) of an operator $\Lambda_x$ that completes the pair $\{L_x, h\}$ into an $\fsl_2$-triple is an open algebraic condition on the classes $x \in \rH^2(X)$. In other words, the dual Lefschetz operator $\Lambda_x$ can be defined for almost all classes $x$, which may lie beyond the ample (or even K\"ahler) classes, and is in fact independent of the complex structure of $X$. Therefore, one can define a Lie algebra $\fg$ over $\QQ$ containing all these operators that is a diffeomorphism invariant of $X(\CC)$. The action of these operators on $\rH(X)$ extends to the whole Lie algebra $\fg$; therefore, by definition, any $\fsl_2$-Lefschetz decomposition factors through $\fg$. Note that the projectivity assumption on $X$ is required only to ensure that the set of $x \in \rH^2(X)$ for which $\Lambda_x$ is defined is a \emph{nonempty} (and thus Zariski-dense) subset.

More generally, we formulate the definition over an arbitrary field $E$ as follows.
\begin{definition}\label{def:LLV}
		Let $X$ be a smooth projective variety over $K$. The \emph{Looijenga--Lunts--Verbitsky (LLV) Lie algebra} $\mathfrak g(X)_{E}$ of $X$ is the smallest $E$-Lie subalgebra of $\mathfrak{gl}(\rH(X))$ generated by all $\fsl_2$-triples
    \[
    \{(L_x, h, \Lambda_x)\}_x\,,
    \]
    where $x\in \rH^2(X)$ is any element satisfying the Hard Lefschetz property.
\end{definition}
\begin{remark}
    The definition of the LLV Lie algebra is valid for any graded Frobenius--Lefschetz algebra, as stated in \cite[Definitions~1.1 and~1.2]{Ver96}. Thus one may also consider the LLV Lie algebra for the Betti cohomology ring of a Kähler manifold.
\end{remark}

\subsubsection{}
The LLV algebra $\fg(X)_{\QQ}$ for the Betti realization $\rH^{\bullet}_{B}(X)$ is a semisimple Lie algebra defined over $\QQ$ (cf.~\cite[(1.9)]{LL97}). Moreover, we have a comparison isomorphism
\[
\fg(X)_{\QQ} \otimes_{\QQ} E \cong \fg(X)_{E}
\]
given in \cite[Example~3.4]{IIKTZ25}. The case where $X$ is a hyper-K\"ahler variety is of greatest interest to us and will be assumed throughout the remainder of this paper. For such a variety, the real form satisfies
\[
\fg(X)_{\QQ} \otimes \RR \cong \fso(4,b_2-2)\,,
\]
by \Cref{prop:ll_isom} below. Since $\fg(X)_{E}$ is invariant under every field embedding $K \hookrightarrow \CC$ when $E = \QQ_{\ell}$, it follows that the $\QQ$-form $\fg(X)_{\QQ}$ is independent of this choice (see \cite[Chapter 5, Theorem 1]{Kne69}).

For simplicity, we abbreviate the notation to $\fg=\fg(X)_E$ when there is no risk of confusion.

\subsubsection{}
The adjoint action of the shifted degree operator $h\in \fg$ induces an eigenspace decomposition of $\fg$.
In the case of hyper-K\"ahler varieties, $\rH(X)$ is of Jordan type as a graded Frobenius--Lefschetz algebra by \cite[Lemma~4.2 and Proposition~4.4]{LL97}, i.e., the eigenspace decomposition for $h$ is of the form \[\fg = \fg_{2} \oplus \fg_{0} \oplus \fg_{-2} \,.\]
In particular, the $0$-eigenspace $\fg_{0}$ is a reductive subalgebra of $\fg$ and admits a decomposition
\[\fg_0=\overline{\fg} \oplus E\cdot h \,,\]
where $\overline{\fg}$ is the semisimple part of $\fg_0$ and satisfies $\overline{\fg} = [\fg_{0},\fg_{0}]$, while the center $\mathfrak z(\fg_0)$ is one-dimensional and spanned by the shifted degree operator $h$. The Lie subalgebra $\overline{\fg}$ is called the \emph{reduced LLV algebra} of $X$.
\subsubsection{}
Note that $\overline{\fg} \subset \fg_{0}$ consists of degree-$0$ operators, and thus the induced $\overline{\mathfrak{g}}$-action on $\rH(X)$ preserves the cohomological degree. In other words, for any integer $0 \leq k \leq 2\dim X$, there is an associated representation
\[\rho_k^{\LLV} \colon \overline{\fg} \longrightarrow \End(\rH^k (X))\,.\]
In particular, the action of $\overline{\fg}$ on $\rH^2(X)$ is obtained by restricting that of $\fg$. Moreover, it acts as a derivation with respect to the cup product:
\begin{equation} \label{eq:ll_respects_cup}
	e\cdot (x \cup y) = (e\cdot x) \cup y + x \cup (e\cdot y), \qquad \text{for} \; e \in \overline{\fg}, \; x, y \in \rH(X)\,.
\end{equation}
One can also see that the action of $\overline{\fg}$ respects the Beauville--Bogomolov--Fujiki form $\overline{q}$ on the second cohomology $\rH^2(X)$, and therefore there is an inclusion
\begin{equation}\label{eq:LLV in so} \overline{\fg}\subset \fso(\rH^2(X),\bar q) \,.
\end{equation}
In fact, the above inclusion \eqref{eq:LLV in so} is an equality so that $\rH^2(X)$ is the standard representation of $\overline{{\mathfrak{g}}}$. More specifically, the following result combines \cite[Theorem~2.3]{Ver96} and \cite[Theorem~4.5]{LL97}.

\begin{theorem}[{\cite[Theorem~2.3]{Ver96}}, {\cite[Theorem~4.5]{LL97}}]\label{prop:ll_isom}
	Let $X$ be a hyper-Kähler variety over $K$. Consider the quadratic space
\begin{equation*}
	\widetilde{\rH}(X) = \rH^2(X) \oplus E^{\oplus 2}, \qquad q = \bar q \oplus \begin{psmallmatrix} 0 & 1 \\ 1 & 0 \end{psmallmatrix} \,,
\end{equation*}
which is called the \emph{Mukai extension} of $\rH^2(X)$ associated with the cohomology of $X$. Then the LLV and reduced LLV Lie algebras of $X$ are as follows:
	\begin{equation}\label{eq:ll_isom_Q}
		\overline{\fg} \cong \fso (\rH^2(X), \overline{q})\,, \quad \fg \cong \fso(\widetilde{\rH}(X), q)\,.
	\end{equation}
    Moreover, let $r = \lfloor b_2(X) / 2 \rfloor$. Then the LLV algebra $\mathfrak{g}$ has type $\mathrm{B}_{r+1}$ or $\mathrm{D}_{r+1}$, depending on the parity of $b_2(X)$, and the reduced algebra $\overline{\mathfrak g}$ has type $\mathrm{B}_r$ or $\mathrm{D}_r$. These algebras are simple apart from the familiar low-rank orthogonal exceptions (in particular, $\mathfrak{so}_4\cong\mathfrak{sl}_2\oplus\mathfrak{sl}_2$). Finally,
	\begin{equation} \label{eq:ll_isom_R}
		\bar{\mathfrak{g}}_{\mathbb{R}} \cong \mathfrak{so}(3, b_2(X) - 3) \,,\quad
        \mathfrak{g}_{\mathbb{R}} \cong \mathfrak{so}(4, b_2(X) - 2)\,.
	\end{equation}
\end{theorem}
\subsubsection{}
Let $\mt(X_{\CC})$ be the Mumford--Tate algebra of the complex projective variety $X_{\CC}$, i.e., the Lie algebra of the Mumford--Tate group $\MT(X_{\CC})$. We write $\overline{\MT}(X_{\CC})$ for the connected special (or Hodge) Mumford--Tate group generated by the restriction of the Hodge morphism to the norm-one torus, and $\overline{\mt}(X_{\CC})$ for its Lie algebra. It is important to note that the Weil operator lies in the LLV Lie algebra $\fg$, as observed by Verbitsky in \cite{Ver96}. Here we recall a refined version given in \cite{GKLR22}.
\begin{theorem}[Verbitsky]
\label{thm:MT in LLV}
	Let $X$ be a hyper-Kähler variety over $K$. Then \[\overline{\mt}(X_{\CC}) \subseteq \overline{\fg}_{\QQ}\] in $\End_{\QQ}(\rH_B(X))$.
\end{theorem}
\begin{proof}
According to \cite[Proposition~2.24]{GKLR22}, the Weil operator for the Hodge structure on $\rH^{\bullet}(X_{\CC},\QQ)$ lies in the real form of the reduced LLV Lie algebra $\overline{\fg}_{\RR}$. By definition, the special Mumford--Tate algebra $\overline{\mt}(X_{\CC})$ is the smallest $\QQ$-Lie algebra whose $\CC$-form contains the Weil operator. Therefore, $\overline{\mt}(X_{\CC}) \subseteq \overline{\fg}_{\QQ}$ as $\QQ$-Lie algebras.
\end{proof}

\begin{example}\label{ex:llv_k3}
	If $X$ is a complex K3 surface, the full cohomology $\rH^{\bullet}(X,\QQ)$ is naturally endowed with the Mukai pairing, which is isomorphic to the Mukai extension $\widetilde{\rH}(X,\QQ)$ of $\rH^2(X,\QQ)$. In this case, the representation-theoretic explanation is clear: $\rH^2(X,\QQ)$ is the standard representation of $\overline{\fg}_{\QQ}$, and $\rH^{\bullet}(X,\QQ)$ is the standard representation of $\fg(X)_{\QQ}$. Also note that $\overline{\fg}_{\QQ}$ can be realized as the special Mumford--Tate algebra of a general nonalgebraic K3 surface, i.e.,
    \[
    \overline{\mt}(X) = \overline{\fg}_{\QQ}\,.
    \]
    The equality also holds when $X$ is a general hyper-Kähler manifold.
\end{example}

\subsection{(Twisted) LLV representations}
Let $X$ be a hyper-Kähler variety over $K$ of dimension $2n$, and let $\fg$ be the LLV Lie algebra of the cohomological realization $\rH(X)$.

\subsubsection{}\label{LLV representation and MT}
Here we recall the integrated forms of the LLV representations introduced by Floccari \cite[\S 2.1]{Flo22b}. The connected algebraic groups over $E$ corresponding to the Lie algebras $\overline{\fg} \subset \fg_0 \subset \fg$ are the Spin groups
\[
\Spin(\rH^2(X),\overline{q}) \subset \GSpin(\rH^2(X),\overline{q}) \subset \Spin(\widetilde{\rH}(X),q)\,.
\]
In what follows, we focus mainly on the action of $\fg_0$, or of its reduced part $\overline{\fg}$, which preserves cohomological degrees.
\begin{notation}\label{notation:Spin groups}
	For simplicity of notation, we set
	\begin{itemize}
		\item $\GSpin = \GSpin(\rH^2(X),\overline{q})$;
		\item $\Spin = \Spin(\rH^2(X),\overline{q})$; and
		\item $\SO = \SO(\rH^2(X),\overline{q})$.
	\end{itemize}
    In the remainder of this paper, we also use the subscript ``$\ell$" for a prime to denote the base change of these groups to the corresponding completion $\QQ_{\ell}$.
 \end{notation}
Under the identifications in \eqref{eq:ll_isom_Q}, the LLV representation \[\rho_{i}^{\LLV} \colon \fg_0 \longrightarrow \End_{E}(\rH^i(X))\] integrates to a group homomorphism
\[
\rho_{i}^{\LLV} \colon \Spin \longrightarrow \GL_{E}\left(\rH^{i}(X)\right).
\]
The kernel $\{1,\delta\}=\mu_2 \subset \Spin$ of the universal covering $\Spin \to \SO$ acts on $\rH^{i}(X)$ via $\rho^{\LLV}_i(\delta) = (-1)^{i}|_{\rH^i(X)}$ (see \cite[Corollary 8.2]{Ver99} and \cite[Theorem A.10.]{KSV19}). When $i = 2k$, the LLV representation $\rho_{2k}^{\LLV}$ factors through the orthogonal group as a representation
\[
\SO \longrightarrow \GL_E(\rH^{2k}(X))\,, \quad \forall 0 \leq k \leq 2n \,,
\]
which is the standard representation of $\SO$ when $k =1$. The full $\SO$-representation need not be faithful in every individual degree. We still denote the resulting representation of $\SO$ by $\rho_{2k}^{\LLV}$.

\subsubsection{}\label{subsub:twisted LLV representation}
In the usual setting of the LLV representation, the degree operator $h \in \fg_0$ acts on $\rH^i(X)$ as the scalar $i - 2n$, which is a shift of the usual cohomological degree by the dimension of $X$. To study the interaction between the LLV representation and various cohomological realizations, it is also convenient to consider the \emph{twisted LLV representation}
\[
\rho^{\tw}_{i}\colon \fg_0^{\tw} \coloneqq \overline{\fg} \oplus E h^{\deg} \longrightarrow \gl_E(\rH^{i}(X))\,
\]
where $h^{\deg} = h + \dim(X)\id$ is the standard degree operator, i.e., $h^{\deg}|_{\rH^i(X)} = i$. As $E$-Lie algebras, $\fg_0^{\tw} \cong \fg_0$, and the integrated representation is
\begin{equation}
	\rho_{i}^{\tw} \colon \GSpin \longrightarrow \GL_E(\rH^i(X))\,,
\end{equation}
such that
\begin{itemize}
	\item $\rho_{i}^{\tw}(z)(x) = z^{i}\cdot x$ for any $z$ in the central torus $\GG_m \subset \GSpin$; and
	\item $\rho_i^{\tw}|_{\Spin} = \rho^{\LLV}_i$.
\end{itemize}
This implies that $\rho_{2k+1}^{\tw}$ is faithful when $\rH^{2k+1}(X) \neq 0$ (see \cite[Lemma~2.6]{Flo22b} and its proof for details).

  \begin{remark}\label{rmk:MT inside GSpin group}
	For a complex hyper-Kähler variety $X$, \Cref{thm:MT in LLV} implies that
	\[
	\MT(X) \subseteq \rho^{\tw}\left(\GSpin(\rH^2(X,\QQ),\overline{q})\right)\,.
	\]
 \end{remark}
The following is a restatement of Verbitsky's theorem (\Cref{thm:MT in LLV}) at the algebraic-group level.
 \begin{lemma}\label{lem:project MT group}
	Suppose that $X$ is a complex hyper-Kähler variety of dimension $2n$.
	\begin{enumerate}
		\item Projection to degree two induces an isomorphism
		\[
		 \pi_2\colon\MT_+(X)\xrightarrow{\sim}\MT_2(X).
		\]
		For $1\leq k\leq2n-1$, projection also induces an isomorphism on special Mumford--Tate groups. Equivalently,
		\[
		\begin{tikzcd}
			\overline{\MT}_+(X) \ar[d, "\simeq"',"\pi_{2k}"]
			  \ar[r]& \SO(\rH^2(X),\overline q)
			  \ar[d,"\rho_{2k}^{\LLV}"]\\
			\overline{\MT}_{2k}(X) \ar[r,hook]&
			  \GL_{\QQ}(\rH^{2k}(X_{\CC},\QQ))
		\end{tikzcd}
		\]
		commutes. The endpoint representations in degrees $0$ and $4n$ are one-dimensional and are deliberately excluded from the latter assertion.
		\item If $\rH^-(X)\neq0$, then
		\[
		 \pi_2\colon\MT(X)\longrightarrow\MT_2(X)
		\]
		is a central isogeny of degree $2$, with kernel $\rho^{\tw}(\mu_2)$, where $\mu_2$ is the central subgroup of $\Spin(\rH^2(X),\overline q)$ acting by parity on total cohomology.
	\end{enumerate}
 \end{lemma}
 \begin{proof}
	By \Cref{rmk:MT inside GSpin group}, the total Mumford--Tate group is contained in the twisted LLV image. In the presentation of $\GSpin$ as the quotient of $\GG_m\times\Spin$ identifying $-1\in\GG_m$ with the parity element $\delta\in\Spin$, the kernels of the twisted representations on even cohomology and on degree two coincide. Consequently, projection from the even twisted LLV image to its degree-two image is an isomorphism. Its restriction to $\MT_+(X)$ is also surjective by the definition of $\MT_2(X)$, proving the first displayed isomorphism. If odd cohomology is nonzero, the total twisted representation is faithful, and projection from its image to degree two has kernel generated by the parity action $\delta|_{\rH^i}=(-1)^i$. This element belongs to $\MT(X)$ as the value at $-1$ of the weight cocharacter. This proves the asserted parity kernel. This is also the purely Hodge-theoretic argument in the proof of \cite[Proposition~6.4 and Lemma~6.5]{FFZ21}; it does not use the motivic splitting or the restriction $b_2\ne3$.

	It remains to justify the individual-degree assertion. Fix an ample class $h\in\rH^2(X,\QQ)$. The special Mumford--Tate group fixes $h$, and for $1\leq k\leq2n-1$ the map
	\[
	 \rH^2(X,\QQ)\longrightarrow\rH^{2k}(X,\QQ),
	 \qquad x\longmapsto x\,h^{k-1},
	\]
	is injective by hard Lefschetz and is equivariant for the special Mumford--Tate group. Hence an element of $\overline\MT_+(X)$ acting trivially in degree $2k$ acts trivially in degree two. Since $\MT_+(X)\to\MT_2(X)$ is an isomorphism, the restriction to special groups is therefore both surjective and injective.
 \end{proof}
\subsubsection{}
From the definition of the LLV Lie algebra $\fg$, it follows immediately that $\fg$ is invariant under deformation, since it is completely determined by the algebra structure of $\rH(X)$. The twisted LLV representation is likewise preserved under deformation.
\begin{lemma}\label{lem:LLV representation is deformation invariant}
	Let $\fX \to S$ be a smooth family of hyper-Kähler varieties over a smooth connected variety $S$. Let $\eta \rightsquigarrow s$ be an étale path of points in $S$. We have the following commutative diagram
	\[
	\begin{tikzcd}[row sep = small]
		& \GL_{\QQ_{\ell}}(\rH_{\ell}^{\bullet}(\fX_{\overline{s}}))\ar[dd,"\specialize_{\eta,s}^*","\simeq"']\\
		\GSpin_{\ell} \ar[ru,"\rho^{\tw}_s"] \ar[rd,"\rho^{\tw}_{\eta}"'] & \\
		& \GL_{\QQ_{\ell}}(\rH_{\ell}^{\bullet}(\fX_{\overline{\eta}}))\mathrlap{\,.}
	\end{tikzcd}
	\]
\end{lemma}
\begin{proof}
	It is sufficient to assume that $\eta \rightsquigarrow s$ is a specialization of points in $S$. The smooth-proper base-change theorem gives a specialization isomorphism
	\[
	\specialize_{\eta,s}\colon
	\rH_{\ell}^{\bullet}(\fX_{\overline{s}})
	\xrightarrow{\sim}
	\rH_{\ell}^{\bullet}(\fX_{\overline{\eta}})
	\]
	of graded $\QQ_\ell$-algebras. Thus it is equivariant for the actions of $\fg_0$ and $\fg_0^{\tw}$.
\end{proof}

\subsection{Motivic lifting of the LLV representation}
In this subsection, we collect some facts on defect groups of hyper-Kähler varieties, which are ingredients in \cite{Flo22,Flo22b,FFZ21}.
\subsubsection{}
For any integer $0 < i < 2\dim X$, there is a natural surjective homomorphism
\[
\pi_i \colon \bG_{\mot}(X) \twoheadrightarrow \bG_{\mot,i}(X)\,.
\]
induced by the inclusion $\fh^i(X) \hookrightarrow \fh(X)$ in $\AM_K$. Under the degree-two projection, there are short exact sequences of motivic Galois groups:
\begin{align}
	\label{eq:exact sequence of motivic Galois groups}
	1 \longrightarrow \bP \longrightarrow &\bG_{\mot}(X) \longrightarrow \bG_{\mot,2}(X) \longrightarrow 1\,, \\
	\label{eq:exact sequence of even motivic Galois groups}
	1 \longrightarrow \bP_+ \longrightarrow &\bG_{\mot,+}(X) \longrightarrow \bG_{\mot,2}(X) \longrightarrow 1\,.
\end{align}
The kernel $\bP$ (resp.~$\bP_+$) is called the \emph{defect group} (resp.~\emph{even defect group}) of $X$.

\begin{lemma}\label{lem:defect group commutes with twisted LLV}
Let $V_E=\rH_B^\bullet(X)\otimes_\QQ E$, let $\Aut_{\alg}^{\gr}(V_E)$ be its algebraic group of $E$-linear graded algebra automorphisms, and put
\[
 \bC_E=\ker\bigl(\Aut_{\alg}^{\gr}(V_E)
 \longrightarrow\GL_E(\rH_B^2(X)\otimes_\QQ E)\bigr).
\]
Then $\bP_E\subseteq\bC_E$, and $\bC_E$ centralizes $\rho^{\tw}(\GSpin_E)$. The analogous statement holds on even cohomology for $\bP_{+,E}$ and
\[
 \bC_{+,E}=\ker\bigl(\Aut_{\alg}^{\gr}(\rH_B^+(X)\otimes_\QQ E)
 \longrightarrow\GL_E(\rH_B^2(X)\otimes_\QQ E)\bigr).
\]
After a Betti--étale comparison identification, the $\ell$-adic kernels $\bP_\ell$ and $\bP_{\ell,+}$ lie in $\bC_{\QQ_\ell}$ and $\bC_{+,\QQ_\ell}$, respectively, and centralize the corresponding twisted LLV images.
\end{lemma}
\begin{proof}
It is enough to work over $\QQ$. The motivic defect group acts by graded algebra automorphisms and trivially in degree two, so it lies in $\bC$. Let $c\in\bC(\overline\QQ)$. For every Lefschetz class $x\in\rH_B^2(X,\overline\QQ)$, one has $cL_xc^{-1}=L_x$. Since $c$ preserves the grading, it commutes with the grading operator; uniqueness of the $\mathfrak{sl}_2$-triple containing $L_x$ then gives $c\Lambda_xc^{-1}=\Lambda_x$. The operators $L_x$ and $\Lambda_x$ generate the LLV algebra, so $c$ centralizes the LLV action. It also commutes with degreewise homotheties and hence with the twisted LLV image. This is the direct argument behind \cite[Lemma~6.8]{FFZ21}. The same proof applies to even cohomology. Finally, Galois acts by graded cup-algebra automorphisms, so its Zariski closure does as well; an element in either $\ell$-adic kernel acts trivially in degree two. The comparison identification therefore gives the final assertions.
\end{proof}
\subsubsection{}
When $b_2(X)\ne3$, \cite[Proposition 4.1]{Flo22} gives a decomposition of motivic Galois groups
\begin{equation}\label{eq:motivic decomposition}
	\bG_{\mot,+}(X) = \bG_{\mot,2}(X) \times \bP_{+}
\end{equation}
given by a splitting $\sigma \colon \bG_{\mot,2}(X) \to \bG_{\mot,+}(X)$ such that $\sigma(\bG_{\mot,2}(X)) = \MT_+(X)$, and $\bG_{\mot,2}(X)$ commutes with $\bP_+$ in $\bG_{\mot,+}(X)$. This implies that the connected component $\bP_+^{\circ}$ is reductive.

The decomposition \eqref{eq:motivic decomposition} implies that, for any
\[
\cM\in\langle\fh^+(X)\rangle^{\otimes},
\]
the invariant part $\cM^{\bP_+}$ belongs to $\langle\fh^2(X)\rangle^{\otimes}$; the invariant object exists because $\bP_+$ is a direct factor of the reductive motivic Galois group. In other words, there is an injective homomorphism in $\AM_K$:
\begin{equation}\label{eq:embedd motive into tensors of H2}
\cM^{\bP_+} \hookrightarrow \bigoplus_{m,n} \fh^2(X)^{\otimes m} \otimes \fh^{2}(X)^{\vee,\otimes n}\,.
\end{equation}
If $\bP_+$ is trivial, then we can take $\cM = \fh^{2k}(X)$, and in particular Conjecture~(Ab) holds for $\fh^{2k}(X)$ by André's theorem.

Moreover, the embedding \eqref{eq:embedd motive into tensors of H2} can be chosen to be compatible with the LLV representation. By Tannakian duality for $\langle \fh^2(X) \rangle^{\otimes}$, \eqref{eq:embedd motive into tensors of H2} gives an injective homomorphism of finite-dimensional $\bG_{\mot,2}(X)$-representations on Betti realizations. By definition of the splitting $\sigma$, the $\bG_{\mot,2}(X)$-action on $\rH_B(\cM^{\bP_{+}})$ is that of $\MT_+(X_{\CC})$, which factors through the twisted LLV representation. Hence we can choose \eqref{eq:embedd motive into tensors of H2} as an injective homomorphism of motives corresponding to an injective homomorphism of $\GSpin$-representations
\[
\rH_B(\cM^{\bP_+}) \hookrightarrow \bigoplus_{m,n} \rH^2_B(X)^{\otimes m} \otimes \rH^2_B(X)^{\vee, \otimes n}\,.
\]

\subsubsection{}\label{sec:almost direct product motivic Galois group}
When $b_2(X)\ne3$, \cite[Theorem~6.9]{FFZ21}\footnote{The published version of \cite{FFZ21} incorrectly claims that $\MT(X_{\CC})$ has trivial intersection with $\bP$ inside $\bG_{\mot}(X)$.} identifies the motivic Galois group of $\fh(X)$ with an almost-direct product of the Mumford--Tate group $\MT(X)$ and the defect group $\bP$:
\[
\bG_{\mot}(X) = \MT(X_{\CC}) \cdot \bP\,,
\]
where $\bP \cap \MT(X_{\CC}) = \{1,\delta\} \subset Z(\bG_{\mot}(X))$ if $X$ has non-vanishing odd-degree cohomology, and $\delta|_{\rH^{i}} = (-1)^i$; otherwise it is truly a direct product.

\begin{lemma}\label{lem:non-trivial center in odd cohomology}
	The element $\delta$ belongs to $Z(\bG_{\ell}(X))\cap \bP_{\QQ_\ell}$ inside $\bG_{\mot}(X)_{\QQ_\ell}$. If $X$ has non-vanishing odd-degree cohomology, then $\delta\neq\id$.
\end{lemma}
\begin{proof}
		By \Cref{lem:homotheties in algebraic monodromy group}, the image of the weight torus lies in $Z(\bG_\ell(X)^\circ)$. Its value at $-1$ acts on $\rH^{i}_{\et}(X_{\overline K},\QQ_\ell)$ as $(-1)^i$, and is therefore $\delta$. If $b_{2k+1}(X)\neq0$, this action is nontrivial on $\rH^{2k+1}_{\et}(X_{\overline K},\QQ_\ell)$.
\end{proof}
\begin{remark}
	If $b_2(X)\ne3$ and the defect group $\bP$ is finite, the Mumford--Tate conjecture holds for $\fh(X)$ (see \cite[Proposition~7.6]{FFZ21}).
\end{remark}
\subsection{The LLV--centralizer trace form}
Let \[ V \coloneqq \rH^{\bullet}_B(X), \quad \bR \coloneqq \rho^{\tw}(\GSpin) \subset \GL_{\QQ}(V). \]
Let $\mathfrak{r} \coloneqq \Lie(\bR)$ and $\mathfrak{c} \coloneqq \Lie(\bC)$. Recall that under the twisted LLV representation, \[ \mathfrak{r} = \overline{\fg} \oplus \QQ h^{\deg},\]
where $\overline{\fg}$ is the reduced LLV Lie algebra and $h^{\deg}$ acts by multiplication by $d$ on $\rH^d_B(X)$.
\begin{lemma}\label{lem:trace orthogonality}
For every field extension $E/\QQ$, $A \in \mathfrak{r}_E$, and $B \in \mathfrak{c}_E$, one has $\Tr_{V_E}(AB)=0$.
\end{lemma}
\begin{proof}
It is enough to prove the assertion over $\QQ$, since it is preserved by scalar extension. We first take $A \in \overline{\fg}$. By \Cref{lem:defect group commutes with twisted LLV}, one has $[A,B] = 0$. For fixed $B$, the functional $A \mapsto \Tr_V(AB)$ vanishes on commutators: cyclicity of the trace together with the Jacobi identity gives $\Tr_V([A_1,A_2]B)=0$. The Lie algebra $\overline{\fg}$ is semisimple and hence perfect (this includes $\mathfrak{so}_3$ when $b_2(X)=3$ and the semisimple, nonsimple algebra $\mathfrak{so}_4$ when $b_2(X)=4$). Therefore, $\Tr_V(AB)=0$ for all $A \in \overline{\fg}$.

It remains to treat $A=h^{\deg}$. Write $\dim X=2n$ and fix an ample class $\eta\in\rH^2_B(X)$. Since $\bC$ acts trivially on $\rH^2_B(X)$, it fixes $\eta$. Since it acts by graded algebra automorphisms, it fixes $\eta^{2n}$ and hence acts trivially on the one-dimensional top cohomology. For $d\leq2n$, consider the nondegenerate Lefschetz pairing
\[
Q_d(x,y)\coloneqq\int_Xx\cup y\cup\eta^{2n-d},
\qquad x,y\in\rH^d_B(X).
\]
The group $\bC$ preserves $Q_d$. The pairing is symmetric for even $d$ and alternating for odd $d$, so the infinitesimal $\bC$-action on $\rH^d_B(X)$ lies respectively in an orthogonal or symplectic Lie algebra. In particular,
\[ \Tr\left(B|_{\rH^d_B(X)}\right) =0 \quad \forall 0 \leq d \leq 2n. \]
For $d >2n$, Hard Lefschetz gives the $\bC$-equivariant isomorphism $L_\eta^{d-2n} \colon \rH^{4n-d}_B(X) \xrightarrow{\sim} \rH^d_B(X)$, so the same trace vanishing holds in every degree. Consequently,
\[ \Tr_V(h^{\deg}B) = \sum_d d \Tr\left(B|_{\rH^d_B(X)}\right) =0.\]
Together with the first part, this proves the assertion.
\end{proof}
\section{Rank estimation of algebraic monodromy groups}
\label{sec:Rank estimation}

Throughout this section, let $X$ be a hyper-Kähler variety over a finitely generated field $K/\QQ$, and fix an embedding $K\hookrightarrow\CC$. We prove the rank comparison using Pink's generation theorem by weak Hodge cocharacters and the graded-algebra centralizer of the twisted LLV representation.
\subsection{Semisimplified monodromy and weak Hodge cocharacters}
\label{subsec:weak Hodge cocharacters}

For each $0\leq i\leq2\dim X$, choose a semisimplification $V_{\ell,i}^{\semisimple}$ of the $G_K$-representation $\rH^i_{\et}(X_{\overline K},\QQ_\ell)$, and put
\[
V_\ell^{\semisimple}\coloneqq
\bigoplus_{i=0}^{2\dim X}V_{\ell,i}^{\semisimple}.
\]
Let $\bG_\ell^{\mathrm{ssm}}(X)$ be the identity component of the Zariski closure of the image of $G_K$ in $\GL_{\QQ_\ell}(V_\ell^{\semisimple})$. This is a connected reductive $\QQ_\ell$-group. The semisimplification is unique up to $G_K$-equivariant isomorphism, so any two choices give linearly conjugate groups. After choosing identifications $V_{\ell,i}^{\semisimple}\simeq\rH^i_{\et}(X_{\overline K},\QQ_\ell)$ degree by degree, we may regard the semisimplified representation as acting on the original graded vector space. All properties used below are invariant under changing these identifications.

\begin{notation}
For a field extension $E/\QQ$ and an algebraic group $\bG/E$, write
\[
X_*(\bG)\coloneqq
\Hom\bigl(\GG_{m,\overline E},\bG_{\overline E}\bigr)
\]
for its set of geometric cocharacters.
\end{notation}

\begin{definition}
\label{def:weak Hodge cocharacter}
Let $E/\QQ$ be a field extension, let $\overline E$ be an algebraic closure of $E$, and let $\rho\colon\bG\to\GL_{\overline E}(V)$ be a finite-dimensional representation of an algebraic group $\bG/\overline E$. Suppose that nonnegative integers $(c_r)_{r\in\ZZ}$ of finite support are prescribed by a Hodge realization, with $\sum_r c_r=\dim_{\overline E}V$. A cocharacter $\lambda\colon\GG_{m,\overline E}\to\bG$ is \emph{weak Hodge} for $(V,(c_r))$ if its weight-$r$ subspace on $V$ has dimension $c_r$ for every $r$. Equivalently, $\rho\circ\lambda$ lies in the $\GL_{\overline E}(V)$-conjugacy class determined by these multiplicities.
\end{definition}

For $V_\ell^{\semisimple}$, the prescribed multiplicity of weight $r$ is
\[
c_r(X)\coloneqq\sum_i h^{r,i-r}(X),
\]
where $h^{p,q}(X)=0$ unless $p,q\geq0$. Thus, following \cite[Definition 3.17(b)]{Pink98}, a cocharacter of $\bG_\ell^{\mathrm{ssm}}(X)_{\overline{\QQ}_\ell}$ is a weak Hodge cocharacter precisely when its weight-$r$ multiplicity on $V_\ell^{\semisimple}\otimes_{\QQ_\ell}\overline{\QQ}_\ell$ is $c_r(X)$ for every $r$. We use Pink's convention: on $\rH_B^i(X)$, weight $p$ occurs with multiplicity $h^{p,i-p}(X)$.

\begin{definition}
\label{def:degreewise weak Hodge cocharacter}
Let $E/\QQ$ be a field extension, let $\overline E$ be an algebraic closure of $E$, and let an algebraic group $\bG/\overline E$ act degree-preservingly on a graded vector space $V=\bigoplus_{i=0}^{2\dim X}V_i$ with $\dim_{\overline E}V_i=b_i(X)$. A cocharacter $\lambda\colon\GG_{m,\overline E}\to\bG$ is \emph{of degreewise Hodge type} with respect to $X$ if, for every $i$ and $r$, the weight-$r$ subspace of $V_i$ has dimension $h^{r,i-r}(X)$.
\end{definition}

The degreewise condition is stronger than Pink's weak Hodge condition on total cohomology, which remembers only the sum of the weight multiplicities over all cohomological degrees. We separate the degrees by a multiplicity-weighted direct sum.

\begin{theorem}[Pink]
\label{thm:Pink weak Hodge generation}
Suppose that $K$ is a number field. Then $\bG_\ell^{\mathrm{ssm}}(X)_{\overline{\QQ}_\ell}$ is generated by the images of its weak Hodge cocharacters, in the sense that it is the smallest closed algebraic subgroup containing those images.
\end{theorem}

\begin{proof}
This is \cite[Theorem 3.18]{Pink98}, together with Pink's remark at the beginning of \cite[\S3]{Pink98} that the results of that section remain valid after replacing one cohomological degree by the semisimplification of a direct sum of cohomology groups. Semisimplification preserves the Frobenius characteristic polynomials of the compatible system, while crystallinity is preserved under subquotients and finite direct sums. Hence Pink's theorem applies to $V_\ell^{\semisimple}$.
\end{proof}

\subsection{A multiplicity-weighted direct sum}
\label{subsec:weighted direct sum}

Choose an integer $M>\max_i b_i(X)$ and consider the semisimple compatible system
\begin{equation}\label{eq:weighted direct sum}
W_{\ell,M}\coloneqq
\bigoplus_{i=0}^{2\dim X}
\left(V_{\ell,i}^{\semisimple}\right)^{\oplus M^i}.
\end{equation}

\begin{lemma}
\label{lem:weighted-direct-sum-monodromy}
Under the diagonal repetition representation
\[
\Delta_M\colon
\bG_\ell^{\mathrm{ssm}}(X)\longrightarrow
\GL_{\QQ_\ell}(W_{\ell,M}),
\]
the connected algebraic monodromy group of $W_{\ell,M}$ is $\Delta_M\bigl(\bG_\ell^{\mathrm{ssm}}(X)\bigr)$. In particular, it is naturally isomorphic to $\bG_\ell^{\mathrm{ssm}}(X)$.
\end{lemma}

\begin{proof}
For every $i$, let
\[
\Delta_{M,i}\colon
\GL_{\QQ_\ell}(V_{\ell,i}^{\semisimple})
\longrightarrow
\GL_{\QQ_\ell}\left(
(V_{\ell,i}^{\semisimple})^{\oplus M^i}
\right)
\]
be the diagonal repetition homomorphism. Taking the product over $i$ gives a homomorphism
\[
\Delta_M\colon
\prod_i\GL_{\QQ_\ell}(V_{\ell,i}^{\semisimple})
\longrightarrow
\GL_{\QQ_\ell}(W_{\ell,M}).
\]
Since every $M^i$ is positive, its restriction to $\bG_\ell^{\mathrm{ssm}}(X)$ is faithful: an element acting trivially on $W_{\ell,M}$ acts trivially on every $V_{\ell,i}^{\semisimple}$, hence on $V_\ell^{\semisimple}$. The restriction is therefore a closed immersion.

The representation on $W_{\ell,M}$ is
\[
G_K\xrightarrow{\rho_\ell^{\semisimple}}
\GL_{\QQ_\ell}(V_\ell^{\semisimple})
\xrightarrow{\Delta_M}
\GL_{\QQ_\ell}(W_{\ell,M}).
\]
Because $\Delta_M$ is a closed immersion,
\[
\overline{\Delta_M\bigl(\rho_\ell^{\semisimple}(G_K)\bigr)}^{\,\mathrm{Zar}}
=
\Delta_M\left(
\overline{\rho_\ell^{\semisimple}(G_K)}^{\,\mathrm{Zar}}
\right).
\]
Passing to identity components proves the assertion.
\end{proof}

\begin{prop}
\label{prop:degreewise-Hodge-type-generation}
Suppose that $K$ is a number field. Then $\bG_\ell^{\mathrm{ssm}}(X)_{\overline{\QQ}_\ell}$ is generated by cocharacters of degreewise Hodge type.
\end{prop}

\begin{proof}
Apply \cite[Theorem 3.18]{Pink98} directly to $W_{\ell,M}$. Finite repetitions of cohomological summands are a formal instance of the direct-sum construction allowed in \cite[\S3]{Pink98}. By \Cref{lem:weighted-direct-sum-monodromy}, its connected algebraic monodromy group is identified with $\bG_\ell^{\mathrm{ssm}}(X)$ through $\Delta_M$.

Let $\lambda$ be a weak Hodge cocharacter of the monodromy group of $W_{\ell,M}$, and pull it back through $\Delta_M$ to a cocharacter of $\bG_\ell^{\mathrm{ssm}}(X)$. Denote by $m_{i,r}(\lambda)$ the multiplicity of weight $r$ on $V_{\ell,i}^{\semisimple}$. The prescribed Hodge multiplicity of weight $r$ on $W_{\ell,M}$ is $\sum_iM^ih^{r,i-r}(X)$, so
\[
\sum_iM^im_{i,r}(\lambda)
=
\sum_iM^ih^{r,i-r}(X).
\]
Both $m_{i,r}(\lambda)$ and $h^{r,i-r}(X)$ lie in $\{0,\ldots,b_i(X)\}$ and hence are strictly smaller than $M$. Uniqueness of base-$M$ expansion gives
\[
m_{i,r}(\lambda)=h^{r,i-r}(X)
\]
for every $i$ and $r$. Thus the pulled-back generators are of degreewise Hodge type, and \Cref{thm:Pink weak Hodge generation} proves the assertion.
\end{proof}

\begin{remark}
Since $\bG_\ell^{\mathrm{ssm}}(X)$ acts degree-preservingly, every cocharacter with image in this group preserves each $V_{\ell,i}^{\semisimple}$. This alone does not make the cocharacter of degreewise Hodge type: Pink's weak Hodge condition on total cohomology fixes only the total multiplicity $\sum_i m_{i,r}(\lambda)=\sum_i h^{r,i-r}(X)$. The multiplicity-weighted representation is used precisely to rule out a redistribution of the weight-$r$ multiplicities among the cohomological degrees.
\end{remark}

\subsection{Levi realizations of the semisimplification}
\label{subsec:Levi realizes semisimplification}

At this stage, no uniqueness of a Levi subgroup is assumed.

\begin{lemma}[Mostow's containment theorem]
\label{lem:Mostow containment}
Let $k$ be a field of characteristic zero, let $G$ be a connected linear algebraic $k$-group, and put $U=R_u(G)$. Then $G$ has a Levi $k$-subgroup, that is, a connected reductive $k$-subgroup $L\subseteq G$ for which multiplication gives an isomorphism $U\rtimes L\xrightarrow{\sim}G$. Any two Levi $k$-subgroups are conjugate by an element of $U(k)$, and every reductive $k$-subgroup $H\subseteq G$ is contained in a Levi $k$-subgroup. More precisely, if $L\subseteq G$ is any fixed Levi $k$-subgroup, then there is an element $u\in U(k)$ such that
\[
 H\subseteq uLu^{-1}.
\]
In particular, every $k$-torus of $G$ is contained in a Levi $k$-subgroup.
\end{lemma}
\begin{proof}
This is Mostow's theorem; see \cite[Chapter~VIII, Theorem~4.3]{Hoc81}.
\end{proof}

\begin{lemma}
\label{lem:Levi realizes semisimplification}
Assume that $\bG_\ell(X)$ is connected. Let $\bL_\ell\subseteq\bG_\ell(X)$ be a Levi subgroup and let $q\colon\bG_\ell(X)\to \bL_\ell$ be the associated Levi projection. For every $i$, the $G_K$-representation obtained from the degree-$i$ representation by composing with $q$ is a semisimplification. Consequently, after a degree-preserving change of basis, $\bG_\ell^{\mathrm{ssm}}(X)$ identifies with $\bL_\ell$.
\end{lemma}

\begin{proof}
Put $G=\bG_\ell(X)$, $U=R_u(G)$, and $L=\bL_\ell$. We first recall that $U$ acts trivially on every irreducible rational representation of $G$. After extending scalars to $\overline{\QQ}_\ell$, the Lie--Kolchin theorem gives a nonzero $U$-fixed vector in every irreducible $G$-module $S$. Since $U$ is normal in $G$, the subspace $S^U$ is $G$-stable, so irreducibility gives $S^U=S$.

Choose a composition series of the $G$-module $\rH^i_{\et}(X_{\overline K},\QQ_\ell)$. Every composition factor factors through $G/U\simeq L$. Since $L$ is reductive in characteristic zero, the restriction of $\rH^i_{\et}(X_{\overline K},\QQ_\ell)$ to $L$ is semisimple and isomorphic to the direct sum of these factors. Therefore the composite
\[
G_K\xrightarrow{\rho_\ell}G\xrightarrow{q}L
\longrightarrow
\GL_{\QQ_\ell}\bigl(\rH^i_{\et}(X_{\overline K},\QQ_\ell)\bigr)
\]
is a semisimplification of the original degree-$i$ representation.

For each $i$, choose an intertwining isomorphism from $V_{\ell,i}^{\semisimple}$ to this $L$-module and take their direct sum. The resulting isomorphism is degree-preserving. Since $\rho_\ell(G_K)$ is Zariski dense in $G$ and $q$ is surjective, $q(\rho_\ell(G_K))$ is Zariski dense in $L$, so this isomorphism conjugates $\bG_\ell^{\mathrm{ssm}}(X)$ onto $L$.
\end{proof}

\begin{cor}
\label{cor:Levi-generated-degreewise-Hodge-type}
Suppose that $K$ is a number field and $K=K^{\conn}$. For every Levi subgroup $\bL_\ell\subseteq\bG_\ell(X)$, the group $\bL_{\ell,\overline{\QQ}_\ell}$ is generated by cocharacters of degreewise Hodge type for its graded action on $\rH^\bullet_{\et}(X_{\overline K},\QQ_\ell)$.
\end{cor}

\begin{proof}
Combine \Cref{prop:degreewise-Hodge-type-generation,lem:Levi realizes semisimplification}. The conjugating isomorphism in \Cref{lem:Levi realizes semisimplification} preserves cohomological degree and hence the degreewise weight multiplicities.
\end{proof}

\subsection{Rigidity of (weak) Hodge cocharacters}
\label{subsec:rigidity of (weak) Hodge cocharacters}

Let $\bR_2$ be the image of $\bR=\rho^{\tw}(\GSpin)$ on $\rH_B^2(X,\QQ)$. The group $\bR_2$ acts by orthogonal similitudes and contains the special orthogonal group $\SO(\rH_B^2(X,\QQ),\overline q)$; it also contains the scalar homotheties coming from the central torus of $\GSpin$.

Fix a representative
\[
\mu_H\colon
\GG_{m,\overline{\QQ}}\longrightarrow
\MT(X_{\CC})_{\overline{\QQ}}
\]
of the geometric Hodge-cocharacter class in Pink's convention. On $\rH_B^i(X)\otimes_{\QQ}\overline{\QQ}$, weight $p$ has multiplicity $h^{p,i-p}(X)$. If $E/\QQ$ is a field extension, we use the same symbol $\mu_H$ after base change to an algebraic closure $\overline E$ along a fixed embedding $\overline{\QQ}\hookrightarrow\overline E$, and write $\mu_{H,2}$ for its degree-two projection. Thus $\mu_{H,2}$ has weights $0,1,2$ with multiplicities $1,b_2(X)-2,1$.

\begin{lemma}
\label{lem:R to R2 finite kernel}
The natural homomorphism $\bR\to\bR_2$ has finite central kernel. Consequently, for every field extension $E/\QQ$ and every algebraic closure $\overline E$ of $E$, the induced map
\[
X_*(\bR_{\overline E})\longrightarrow X_*(\bR_{2,\overline E})
\]
is injective.
\end{lemma}

\begin{proof}
The group $\bR$ is connected because it is the image of the connected group $\GSpin$. It is enough to prove that the differential of $\bR\to\bR_2$ is injective. Under the decomposition $\mathfrak r=\overline{\fg}\oplus\QQ h^{\deg}$, the differential is
\[
A+a h^{\deg}\longmapsto
A+2a\,\id_{\rH_B^2(X)}.
\]
If this is zero, taking traces gives $2ab_2(X)=0$, since $A\in\fso(\rH_B^2(X),\overline q)$ has trace zero. Thus $a=0$ and then $A=0$. The kernel is finite; being a finite normal subgroup of the connected group $\bR$, it is central. A cocharacter with image in a finite group is trivial, which gives injectivity on cocharacters after scalar extension.
\end{proof}

\begin{lemma}
\label{lem:degree two Hodge cocharacter rigidity}
Let $E/\QQ$ be a field extension, let $\overline E$ be an algebraic closure of $E$ equipped with an embedding $\overline{\QQ}\hookrightarrow\overline E$, and let $N\geq1$. Suppose that
\[
\mu\colon\GG_{m,\overline E}\longrightarrow\bR_{2,\overline E}
\]
has weights $0,N,2N$ on $\rH_B^2(X,\QQ)\otimes_\QQ\overline E$, with multiplicities $1,b_2(X)-2,1$. Then $\mu$ and $N\mu_{H,2}$ are conjugate under $\bR_2(\overline E)$, where $(N\mu_{H,2})(t)=\mu_{H,2}(t^N)$.
\end{lemma}

\begin{proof}
Put $V_2=\rH_B^2(X,\QQ)\otimes_\QQ\overline E$ and write $V_2=V_0\oplus V_N\oplus V_{2N}$ for the $\mu$-weight decomposition. Let $\nu\colon\bR_{2,\overline E}\to\GG_{m,\overline E}$ be the similitude character and write $(\nu\circ\mu)(t)=t^c$. If $x\in V_a$ and $y\in V_b$, then
\[
\overline q(\mu(t)x,\mu(t)y)
=t^{a+b}\overline q(x,y)
=t^c\overline q(x,y).
\]
Thus $\overline q(V_a,V_b)\neq0$ only if $a+b=c$. Nondegeneracy implies that every nonzero weight space $V_a$ pairs perfectly with $V_{c-a}$; thus $V_{c}$ and $V_{c-2N}$ are both nonzero. The only possibility is $c=2N$.

It follows that $V_0$ and $V_{2N}$ are isotropic lines paired perfectly by $\overline q$, while $V_N=(V_0\oplus V_{2N})^\perp$ is nondegenerate. Choose $0\neq e\in V_0$ and $f\in V_{2N}$ with $\overline q(e,f)=1$. Make analogous choices $e'$ and $f'$ for the weight decomposition of $N\mu_{H,2}$. The map $e\mapsto e'$, $f\mapsto f'$ is an isometry between the two hyperbolic planes, and Witt's extension theorem extends it to an element $g\in\OO(V_2,\overline q)$. Necessarily $g(V_N)=V'_N$, so
\[
g\mu(t)g^{-1}=N\mu_{H,2}(t)
\]
for every $t$.

If $\det(g)=-1$, choose a nonisotropic vector $v\in V'_N$. Such a vector exists because $V'_N$ is nondegenerate and has dimension $b_2(X)-2\geq1$. The orthogonal reflection $s_v$ has determinant $-1$ and commutes with $N\mu_{H,2}$, which acts on $V'_N$ by the scalar $t^N$. Replacing $g$ by $s_vg$, we obtain $g\in\SO(V_2,\overline q)\subseteq\bR_2(\overline E)$.
\end{proof}

\begin{cor}
\label{cor:lift degree two conjugacy}
Let $E/\QQ$ be a field extension and let $\overline E$ be an algebraic closure of $E$. If two cocharacters of $\bR_{\overline E}$ have conjugate images in $\bR_{2,\overline E}$, then they are conjugate under $\bR(\overline E)$.
\end{cor}

\begin{proof}
By definition of $\bR_2$, the map $\bR\to\bR_2$ is surjective. It is therefore surjective on $\overline E$-points, so a conjugating element in $\bR_2(\overline E)$ lifts to $\bR(\overline E)$. After this conjugation, the two cocharacters have the same degree-two projection and hence are equal by \Cref{lem:R to R2 finite kernel}.
\end{proof}

The following proposition replaces the motivic defect splitting used in \cite{FFZ21} by a statement internal to the graded cohomology algebra.

\begin{prop}
\label{prop:defect rigidity weak Hodge}
Let $E/\QQ$ be a field extension and let $\overline E$ be an algebraic closure of $E$ equipped with an embedding $\overline{\QQ}\hookrightarrow\overline E$. Let
\[
\lambda\colon\GG_{m,\overline E}\longrightarrow
\Aut_{\alg}^{\gr}(V_{\overline E})
\]
be a cocharacter of degreewise Hodge type. If its degree-two projection factors through $\bR_{2,\overline E}$, then $\lambda$ is conjugate to $\mu_H$ under $\bR(\overline E)$. In particular, it factors through $\bR_{\overline E}$.
\end{prop}

\begin{proof}
The degree-two weights of $\lambda$ are $0,1,2$ with multiplicities $1,b_2(X)-2,1$ because it is of degreewise Hodge type. By \Cref{lem:degree two Hodge cocharacter rigidity}, its degree-two projection is conjugate to $\mu_{H,2}$ under $\bR_2(\overline E)$. Lift a conjugating element to $\bR(\overline E)$ and conjugate $\lambda$. We may therefore assume that $\lambda_2=\mu_{H,2}$.

Put
\[
 A=d\lambda(1),\qquad H=d\mu_H(1),\qquad B=A-H
\]
in $\End_{\overline E}(V_{\overline E})$. The equality in degree two gives $B\in\mathfrak c_{\overline E}$. By \Cref{lem:defect group commutes with twisted LLV}, $B$ commutes with $H\in\mathfrak r_{\overline E}$. Moreover, \Cref{lem:trace orthogonality} gives $\Tr_V(HB)=0$.

The cocharacters $\lambda$ and $\mu_H$ have the same weights in every degree, so $\Tr_V(A^2)=\Tr_V(H^2)$. Since $A=H+B$ and $[H,B]=0$, it follows that
\[
 \Tr_V(B^2)=0.
\]
The commuting semisimple endomorphisms $A$ and $H$ are simultaneously diagonalizable. Their eigenvalues are integers, and hence the eigenvalues of $B=A-H$ are integers. Thus $\Tr_V(B^2)$ is a sum of squares of integers with nonnegative multiplicities; it can vanish only when $B=0$. The two cocharacters consequently have the same differential and are equal in $\GL(V_{\overline E})$. Undoing the conjugation proves the claim.
\end{proof}

\subsection{Uniqueness of Levi factors over number fields}
\label{subsec:number field Levi rigidity}

After the uniform finite extension supplied by \Cref{thm:independence of connected components}, we may assume $K=K^{\conn}$.

\begin{theorem}
\label{thm:Levi rigidity number field}
Suppose that $K$ is a number field and $K=K^{\conn}$. Then every Levi subgroup $\bL_\ell\subseteq\bG_\ell(X)$ satisfies
\[
\bL_\ell=\MT(X_{\CC})_{\QQ_\ell}
\]
inside the chosen motivic realization $\bG_{\mot}(X)_{\QQ_\ell}$.
\end{theorem}

\begin{proof}
By \Cref{cor:Levi-generated-degreewise-Hodge-type}, the group $\bL_{\ell,\overline{\QQ}_\ell}$ is generated by cocharacters of degreewise Hodge type. The $\ell$-adic monodromy group acts by graded cup-algebra automorphisms, so these cocharacters lie in $\Aut_{\alg}^{\gr}(V_{\overline{\QQ}_\ell})$. The degree-two projection of each generator lies in
\[
 \bG_{\ell,2}(X)_{\overline{\QQ}_\ell}
 =\MT_2(X_{\CC})_{\overline{\QQ}_\ell}
 \subseteq\bR_{2,\overline{\QQ}_\ell},
\]
where the equality is \Cref{thm:Mumford--Tate conjecture HK degree 2}. By \Cref{prop:defect rigidity weak Hodge}, every generator lies in $\bR_{\overline{\QQ}_\ell}$. Hence $\bL_\ell\subseteq\bR_{\QQ_\ell}$ by descent.

Let $U_\ell=R_u(\bG_\ell(X))$. Its image in $\bG_{\ell,2}(X)$ is a connected normal unipotent subgroup. The degree-two Mumford--Tate theorem gives
\[
\bG_{\ell,2}(X)=\MT_2(X_{\CC})_{\QQ_\ell},
\]
which is reductive, so $U_\ell$ maps trivially to degree two. Since $\bG_\ell(X)=U_\ell\rtimes \bL_\ell$ and the degree-two projection is surjective, the restriction $\bL_\ell\to\bG_{\ell,2}(X)$ is surjective.

Let $q\colon\bR_{\QQ_\ell}\to\bR_{2,\QQ_\ell}$ be the degree-two projection. It has finite kernel by \Cref{lem:R to R2 finite kernel}. Both $\bL_\ell$ and $\MT(X_{\CC})_{\QQ_\ell}$ are connected subgroups of $\bR_{\QQ_\ell}$ and map onto $\bG_{\ell,2}(X)=\MT_2(X_{\CC})_{\QQ_\ell}$. They are therefore connected subgroups of
\[
 Q=\bigl(q^{-1}(\bG_{\ell,2}(X))\bigr)^\circ
\]
having the same dimension as $Q$. Thus both equal $Q$.
\end{proof}

\begin{remark}
The theorem does not assume uniqueness of Levi factors. It proves that, over a number field, every Levi subgroup equals the same subgroup $\MT(X_{\CC})_{\QQ_\ell}$ in the chosen motivic realization.
\end{remark}

\subsection{Rank comparison}
\label{subsec:Rank bounds for hyper-Kahler varieties}

For any algebraic group $G$, write $\rk G\coloneqq\rk G^\circ$; equivalently, this is the dimension of a maximal torus of $G^\circ$ after base change to an algebraic closure.

\begin{lemma}
\label{lem:rank under surjection}
Let $f\colon G\to H$ be a surjective homomorphism of connected algebraic groups over a field of characteristic zero. Then $\rk G\geq\rk H$. If
\[
1\longrightarrow Z\longrightarrow G\longrightarrow H\longrightarrow1
\]
is exact, then $\rk G=\rk Z^\circ+\rk H$.
\end{lemma}

\begin{proof}
After base change to an algebraic closure, the image of a maximal torus of $G$ is a maximal torus of $H$. If $T_0$ is a maximal torus of $Z^\circ$, choose a maximal torus $T\subseteq G$ containing $T_0$. Then $T_0$ is the identity component of the kernel of $T\to f(T)$, and the dimension formula for tori gives the assertion.
\end{proof}

For the total and even representations, respectively, put
\[
\bP_\ell\coloneqq
\ker\bigl(\bG_\ell(X)\to\bG_{\ell,2}(X)\bigr),
\qquad
\bP_{\ell,+}\coloneqq
\ker\bigl(\bG_{\ell,+}(X)\to\bG_{\ell,2}(X)\bigr).
\]

\begin{theorem}
\label{thm:rank of ell algebraic monodromy groups}
Let $X$ be a hyper-Kähler variety over a finitely generated field $K/\QQ$. For every rational prime $\ell$,
\[
\rk\bG_\ell(X)
=
\rk\bG_{\ell,+}(X)
=
\rk\bG_{\ell,2}(X).
\]
Moreover,
\[
\bP_\ell^\circ=R_u(\bG_\ell(X)^\circ),
\qquad
\bP_{\ell,+}^\circ=R_u(\bG_{\ell,+}(X)^\circ).
\]
\end{theorem}

\begin{proof}
Ranks and identity components are unchanged after a finite extension, so we may replace $K$ by $K^{\conn}$.

Assume first that $K$ is a number field. Let $\bL_\ell$ be a Levi subgroup of $\bG_\ell(X)$. By \Cref{thm:Levi rigidity number field}, $\bL_\ell=\MT(X_{\CC})_{\QQ_\ell}$. A connected algebraic group and any Levi subgroup have the same rank. Moreover, \Cref{lem:project MT group} shows that $\MT(X_{\CC})\to\MT_2(X_{\CC})$ has finite kernel, while the degree-two Mumford--Tate conjecture is known for $X$. Hence
\[
\rk\bG_\ell(X)
=
\rk\MT(X_{\CC})
=
\rk\MT_2(X_{\CC})
=
\rk\bG_{\ell,2}(X).
\]

Now suppose that $K/\QQ$ has positive transcendence degree. Use the spreading out $\fX \to B$ from \Cref{reduction to number field}. After shrinking $B$, its relevant characteristic-zero fibers are polarized hyper-Kähler varieties. For the fixed prime $\ell$, choose a closed point $s\in B$ for which cospecialization identifies the connected total monodromy groups:
\[
\bG_\ell(\fX_s)^\circ
\simeq
\bG_\ell(X)^\circ.
\]
The degree-two monodromy group of $\fX_s$ embeds into that of the generic fiber. Applying the number-field case to $\fX_s$ gives
\[
\rk\bG_\ell(X)
=
\rk\bG_\ell(\fX_s)
=
\rk\bG_{\ell,2}(\fX_s)
\leq
\rk\bG_{\ell,2}(X).
\]
The reverse inequality follows from the degree-two quotient and \Cref{lem:rank under surjection}. Thus the total and degree-two ranks are equal. Since the even monodromy group is an intermediate quotient, its rank is the same.

Finally, apply \Cref{lem:rank under surjection} to the two degree-two projections. Rank equality shows that $\bP_\ell^\circ$ and $\bP_{\ell,+}^\circ$ have rank zero. A connected linear algebraic group of rank zero in characteristic zero is unipotent. Since the degree-two quotient is reductive by the degree-two Mumford--Tate theorem, the unipotent radical of each source lies in the corresponding kernel; conversely, each connected kernel is a normal unipotent subgroup. This proves the asserted equalities.
\end{proof}

\begin{cor}
\label{cor:Levi subgroup isogenous to G2}
For every finitely generated $K/\QQ$ and every Levi subgroup $\bL_\ell\subseteq\bG_\ell(X)^\circ$, the degree-two projection induces an isogeny
\[
\bL_\ell\longrightarrow\bG_{\ell,2}(X)^\circ.
\]
It also induces isogenies on the identity components of the centers and on the derived subgroups. The analogous statements hold for the even monodromy group.
\end{cor}

\begin{proof}
The unipotent radical maps trivially to the reductive degree-two quotient, so the degree-two projection restricts to a surjection on a Levi subgroup. By \Cref{thm:rank of ell algebraic monodromy groups}, the source and target have the same rank. A surjective homomorphism of connected reductive groups of equal rank has finite kernel and hence is an isogeny. The assertions for connected centers and derived groups follow from the standard structure theory of reductive groups. The same argument applies to the even monodromy group.
\end{proof}

\section{Mumford--Tate conjecture for hyper-Kähler varieties}
\label{sec:Mumford--Tate conjecture for HK varieties}
Let $X$ be a hyper-Kähler variety over a finitely generated field $K/\QQ$. In this section, we focus on the proof of our main results for $X$. As before, we fix a field embedding $K \hookrightarrow \CC$. After the uniform finite extension of \Cref{thm:independence of connected components}, we may and do assume that $K=K^{\conn}$ for $\fh(X)$. This replacement does not change any connected algebraic monodromy group or any conclusion below. The argument is arranged in dependency order. We first compare derived monodromy in a suitable family, then identify the canonical Levi factor, and finally deduce the even splitting and the semisimple Mumford--Tate conjecture. Semisimplicity itself is treated only after these structural results are established.
\subsection{Monodromy in families of hyper-Kähler varieties}
A fruitful approach to the Mumford--Tate conjecture, inspired by \cite{Andre96} and \cite{Moonen17a}, is to reduce it to a simpler situation using a smooth family with ``large monodromy.'' We now recall several basic properties of the monodromy groups associated with families of hyper-Kähler varieties.
\subsubsection{Geometric monodromy}
Let $B$ be a geometrically connected smooth variety over $K$. Fix a field embedding $K \subset \CC$ and a $\QQ$VHS $\VV$ on $B_{\CC} = B \times_{K} \CC$. Recall that a complex point $s \in B(\CC)$ is \emph{Hodge generic} (with respect to $\VV$) if the Mumford--Tate group $\MT(\VV_{s})$ is maximal under monodromy conjugation. As a $\QQ$-local system, the \emph{algebraic monodromy group} of $\VV$ (with a base point $s \in B(\CC)$) is the identity component of the Zariski closure of
\[
\rho \colon \pi_1(B_{\CC},s) \longrightarrow \GL_{\QQ}(\VV_s)\,,
\]
denoted by $\bM(\VV)$. If one chooses a different base point, the monodromy representation changes by conjugation via parallel transport, and therefore the algebraic monodromy group remains isomorphic. In what follows, we sometimes suppress the base point when it causes no ambiguity.
\subsubsection{Comparison and normality}\label{subsub:generic Mumford-Tate group}
Let us recall some general features of $\bM(\VV)$:
\begin{itemize}[--]
	\item It is a normal subgroup $\bM(\VV)\triangleleft\MT(\VV_t)^{\der}$ if $t$ is Hodge generic in $B(\CC)$; see \cite[Theorem 1]{Andre92}. If the variation contains a CM fiber, André's fixed-part theorem identifies $\bM(\VV)=\MT(\VV_t)^{\der}$ for a Hodge-generic point $t$.
		\item If there is a $\QQ_{\ell}$-local system $\VV_{\ell}$ on $B$ such that $\VV \otimes \QQ_{\ell} \cong \VV_{\ell}^{\an}$, then it induces an isomorphism of $\QQ_{\ell}$-algebraic groups:
	\begin{equation}\label{eq:Artin comp monodromy groups}
		\bM(\VV)_{\QQ_{\ell}} \cong \bM(\VV_{\ell})\,.
	\end{equation}
		See \cite[Lemma~3.3]{Urb22} for details. Here $\bM(\VV_{\ell})$ is the identity component of the Zariski closure of the image of \[\pi_1^{\et}(B_{\overline{K}}, s) \to \GL_{\QQ_{\ell}}(\VV_s \otimes \QQ_{\ell})\,.\]
\end{itemize}
\begin{remark}
Motivic Galois groups of André motives in a family share features similar to those of Mumford--Tate groups. Let $B$ be a geometrically connected smooth variety over $K \subseteq \CC$, and let $\cM/B$ be a family of motives over $B$. Assume that the generic motivic Galois group $\bG_{\mot}(\cM/B)$ is connected. Then the algebraic monodromy group of its Betti realization satisfies
\[
\bM(\rH_B(\cM/B)) \triangleleft \bG_{\mot}(\cM/B)^{\der}
\]
as a normal subgroup. If the family contains a fiber with abelian motivic Galois group, then $\bM(\rH_B(\cM/B))=\bG_{\mot}(\cM/B)^{\der}$ for the generic motivic Galois group. See \cite[Theorem 0.6.4]{Andre96b}.
\end{remark}
\subsubsection{Galois-generic fibers}
Let $\VV_{\ell}$ be an $\ell$-adic local system on $B$. Consider the subset of points of $B$ whose $\ell$-adic algebraic monodromy group is strictly smaller than that of the generic fiber under the specialization $\eta \rightsquigarrow s$, i.e.,
\[
\Exc_{\ell} \coloneqq \Set*{s \in B \given \bG_{\ell,s} \subsetneq \bG_{\ell,\eta} }\,.
\]
The subset $\Exc_{\ell}$ is called the ($\ell$-adic) exceptional locus of $\VV_{\ell}$, and any point $s\in B\setminus\Exc_{\ell}$ is called an \emph{($\ell$-adic) Galois-generic point}. By definition, the generic point $\eta\in B$ is Galois generic for every $\ell$-adic local system on $B$.

\begin{example}\label{lem:Hodge generic iff Galois generic}
    Let $f \colon\fX \to B$ be a smooth projective family of hyper-Kähler varieties. For the degree-two primitive-cohomology local system $\VV_{\ell,2} \coloneqq R^2_{\prim}f_*\QQ_{\ell}$, a point $s \in B$ is Galois generic with respect to $\VV_{\ell,2}$ if and only if, for any field embedding $k(s) \hookrightarrow \CC$, the point $s_{\CC} \in B(\CC)$ is Hodge generic with respect to $\VV_2 \coloneqq R^2_{\prim}f_{\CC,*} \QQ$. This follows from $(\MTC_2)$ for hyper-Kähler varieties (\Cref{thm:Mumford--Tate conjecture HK degree 2}).
\end{example}
\begin{remark}
    Our definition of the exceptional locus coincides with the traditional one based on the dimension of the connected $\ell$-adic monodromy groups: a proper closed subgroup of a connected algebraic group that is itself connected has strictly smaller dimension.
\end{remark}

\begin{lemma}\label{lem:monodromy inside derived when Galois generic}
Let $f\colon\fX\to B$ be a smooth projective family of hyper-Kähler varieties, and let $\square\in\{\bullet,+\}$. Suppose that $s\in B$ is $\ell$-adic Galois generic in degree two. After identifying cohomology along a specialization path, every Levi subgroup $\bL_{\ell,\square,s}$ is contained in a Levi subgroup $\bL_{\ell,\square,\eta}$, and in fact
\[
 \bL_{\ell,\square,s}=\bL_{\ell,\square,\eta}.
\]
Moreover, the geometric monodromy group satisfies
\[
 \bM(\VV_{\ell,\square})
 \subseteq \bL_{\ell,\square,s}^{\der}
\]
as a connected normal semisimple subgroup.
\end{lemma}
\begin{proof}
Cospecialization gives $\bG_{\ell,\square,s}\subseteq\bG_{\ell,\square,\eta}$. By \Cref{lem:Mostow containment}, the reductive subgroup $\bL_{\ell,\square,s}$ is contained in some generic Levi subgroup $\bL_{\ell,\square,\eta}$. By \Cref{cor:Levi subgroup isogenous to G2}, both Levi subgroups are isogenous to their degree-two groups. Since $s$ is degree-two Galois generic,
\[
 \dim \bL_{\ell,\square,s}=\dim\bG_{\ell,2,s}
 =\dim\bG_{\ell,2,\eta}=\dim \bL_{\ell,\square,\eta}.
\]
The inclusion is therefore an equality.

The geometric monodromy group is connected semisimple and normal in the generic arithmetic monodromy group. A connected semisimple normal subgroup lies in every Levi: its intersection with the unipotent radical is trivial, and the commutator with that radical lies in the intersection. It consequently lies in the derived subgroup of the common Levi.
\end{proof}

Let $\VV$ be a polarizable $\QQ\VHS$ on $B$. If the algebraic monodromy group $\bM(\VV) = \MT(\VV_s)^{\der}$ for Hodge generic points $s$, then we say $\VV$ has \emph{maximal monodromy}.

\begin{definition}
    Let $f \colon \fX \to B$ be a smooth projective family. We say $f$ has maximal monodromy (resp.~maximal monodromy in degree $i$) if $R^{\bullet}f_*\QQ$ (resp.~$R^if_*\QQ$) has maximal monodromy.
\end{definition}
\subsubsection{Derived monodromy at a Galois-generic fiber}
As recalled in \Cref{subsub:generic Mumford-Tate group}, the algebraic monodromy group of a polarized $\QQ$-variation of Hodge structure is a normal subgroup of the derived generic Mumford--Tate group; this follows from Deligne's theorem of the fixed part. For a family of hyper-Kähler varieties, this determines the derived $\ell$-adic monodromy group in degree $2$ at a Galois-generic fiber, using the degree-two Mumford--Tate conjecture.

By the rank estimate, the derived subgroups of Levi factors of the total and even monodromy groups are isogenous to the derived degree-two group (\Cref{cor:Levi subgroup isogenous to G2}). We now describe these maps at $\ell$-adic Galois-generic fibers.
\begin{prop}\label{prop:generic derived splitting}
Let $f \colon \fX \to B$ be a smooth projective family of hyper-Kähler varieties over a smooth geometrically connected variety $B/K$, with maximal monodromy in degree $2$, and fix a prime $\ell$. Let $s\in B\setminus\Exc_{\ell,2}$ and let $\bL_{\ell,\square,s}$ be a Levi subgroup of the monodromy group of $\fX_s$, where $\square\in\{\bullet,+\}$. Then
\[
 \MT_{\square}(\fX_{s_\CC})_{\QQ_\ell}^{\der}
 =\bM(\VV_{\ell,\square})
 =\bL_{\ell,\square,s}^{\der}
\]
inside $\bG_{\mot,\square}(\fX_s)_{\QQ_\ell}$.
\end{prop}
\begin{proof}
	By \Cref{lem:monodromy inside derived when Galois generic}, $\bM(\VV_{\ell,\square})$ is contained in $\bL_{\ell,\square,s}^{\der}$. Maximal monodromy in degree two and $(\MTC_2)$ give
	\[
	 \bM(\VV_{\ell,2})
	 =\MT_2(\fX_{s_\CC})_{\QQ_\ell}^{\der}
	 =\bG_{\ell,2,s}^{\der}.
	\]
	The projections from the even and total geometric monodromy groups to degree two are, respectively, an isomorphism and a central isogeny of degree at most $2$; this follows from \Cref{lem:project MT group} and Artin comparison. On the other hand, \Cref{cor:Levi subgroup isogenous to G2} makes $\bL_{\ell,\square,s}^{\der}\to\bG_{\ell,2,s}^{\der}$ an isogeny. Thus the contained connected semisimple groups have the same dimension, so
	\[
	 \bM(\VV_{\ell,\square})=\bL_{\ell,\square,s}^{\der}.
	\]
	The Hodge-theoretic maximal-monodromy equality gives the remaining identification with the derived Mumford--Tate group.
\end{proof}

Combining the (local) Torelli theorem for hyper-Kähler varieties with the descent results in \cite{Andre96}, we obtain the following family.
\begin{prop}\label{prop:to be a Galois generic point in a fiber}
Let $X$ be a hyper-Kähler variety over $K$. Fix a prime $\ell$. After a finite extension $K'/K$, there is a smooth projective family $\fX\to B$ of hyper-Kähler varieties over a smooth geometrically connected $K'$-variety, with maximal monodromy in degree two, and a point $b\in B(K')\setminus\Exc_{\ell,2}$ such that $\fX_b\cong X_{K'}$.
\end{prop}
\begin{proof}
	If $b_2(X)=3$, take $K'=K$, $B=\Spec K$, and the constant family $\fX=X$. The degree-two Mumford--Tate group is a torus, so both its derived group and the geometric monodromy group of this constant family are trivial. Thus the family has maximal monodromy in degree two and its unique point is Galois generic.

	Assume $b_2(X)>3$. After a finite extension, choose a fine moduli space of polarized hyper-Kähler varieties with neat level structure, an étale chart through the point of $X$, and its universal family. The period map on this chart is étale by local Torelli. Let $\mathscr S_b$ be the smallest Mumford--Tate special subvariety of the orthogonal Shimura variety that contains the period of $X$, and let $B$ be the connected component through $b$ of its inverse image, shrunk if necessary.

	By minimality of $\mathscr S_b$, the point $b$ is Hodge generic for $R^2_{\prim}f_*\QQ$ on $B$. Special points are dense in each connected component of $\mathscr S_b$, so the open period image of $B$ contains a CM point. André's fixed-part theorem therefore gives
	\[
	 \bM(R^2_{\prim}f_*\QQ)
	 =\MT_2(\fX_{b,\CC})^{\der};
	\]
	thus the family has maximal monodromy in degree two. The algebraicity and descent after a finite extension are supplied by the descent theorem in \cite{Andre96}; see also \cite[Theorem 4.5.2]{Bindt21}. Finally, $(\MTC_2)$ identifies Hodge genericity with $\ell$-adic Galois genericity, so $b\notin\Exc_{\ell,2}$.
\end{proof}

\begin{cor}\label{cor:derived parts are equal}
For a hyper-Kähler variety $X/K$, the derived subgroup of every Levi $\bL_\ell\subseteq\bG_\ell(X)^\circ$ satisfies
\[
 \bL_\ell^{\der}=\MT(X_{\CC})_{\QQ_\ell}^{\der}
 \subseteq\bG_{\mot}(X)_{\QQ_\ell}.
\]
\end{cor}
\begin{proof}
Apply \Cref{prop:to be a Galois generic point in a fiber} and \Cref{prop:generic derived splitting} after the finite extension supplied there. Connected monodromy, Levi derived groups, and the Mumford--Tate group are unchanged by that extension.
\end{proof}

\subsubsection{Generic monodromy of a universal family}
In this subsubsection, assume $b_2(X)>3$. Let $\fX\to B$ be the pullback of a universal polarized family to a finite étale cover of a connected component of the corresponding moduli space. Put $\Lambda_h=h^\perp\subset\rH^2(\fX_{\eta_\CC},\ZZ)$. The transcendental Hodge structure $T(\fX_{\eta_\CC})$ is simple by \cite[Theorem~1.4.1]{Zarhin83}. Moreover, local Torelli gives full generic monodromy $\SO(\Lambda_h)_\QQ$; see \cite[Corollary~3.3.3]{Andre96b}. It follows from Zarhin's description of the Hodge group \cite[Theorem~2.2.1]{Zarhin83} that
\begin{equation}\label{eq:generic transcendental endomorphisms}
 T(\fX_{\eta_\CC})\cong\Lambda_h,
 \qquad
 \End_{\Hdg}\bigl(T(\fX_{\eta_\CC})\bigr)=\QQ.
\end{equation}

\subsection{Canonical Levi factors}
We show that every Levi subgroup of the total monodromy group factors through the twisted LLV representation. For the $b_2>3$ branch below, this follows from \eqref{eq:generic transcendental endomorphisms} and specialization in a universal polarized family. An étale path from the given fiber to a geometric generic point identifies the graded $\ell$-adic cohomology spaces and gives the usual cospecialization inclusion of connected algebraic monodromy groups. Under the same identification, the twisted LLV representation is constant by \Cref{lem:LLV representation is deformation invariant}. We use only these two consequences; no claim that an arbitrarily chosen parallel-transport operator is itself a motivated correspondence is needed.

\begin{prop}\label{prop:cocharacters in Spin group}
	Keep the notation of \Cref{notation:Spin groups}. For a hyper-Kähler variety $X/K$, every Levi subgroup $\bL_\ell\subseteq\bG_\ell(X)^\circ$ satisfies
	\[
	 \bL_\ell\subseteq\rho^{\tw}(\GSpin_\ell).
	\]
	The analogous assertion holds for every Levi subgroup of $\bG_{\ell,+}(X)^\circ$.
\end{prop}
\begin{proof}
	Suppose first that $b_2(X)=3$. Apply \Cref{lem:b2=3 descends to number field}. The equality of Galois images identifies the connected monodromy groups of $X$ and $X_0$. Hence \Cref{thm:Levi rigidity number field} gives
	\[
	 \bL_\ell=\MT(X_\CC)_{\QQ_\ell}
	 \subseteq\rho^{\tw}(\GSpin_\ell)
	\]
	for every total Levi.

	Its even image is a Levi subgroup of $\bG_{\ell,+}(X)^\circ$. By the rank theorem, the unipotent radical of the latter group is $\bP_{\ell,+}^\circ$, which centralizes the twisted LLV image by \Cref{lem:defect group commutes with twisted LLV}. Since every even Levi subgroup is conjugate to the displayed one by this radical, all even Levi subgroups coincide with it. This proves both assertions when $b_2=3$.

	Assume now that $b_2(X)>3$. After a finite extension, place $X$ in a universal polarized family $\fX\to B$ and write $s$ for its point and $\eta$ for the generic point. By \Cref{lem:Mostow containment}, the central weight torus is contained in a generic Levi subgroup $\bL_{\ell,\eta}$. By \eqref{eq:generic transcendental endomorphisms}, the generic endomorphism field is $\QQ$. By \Cref{cor:derived parts are equal},
	\[
	 \bL_{\ell,\eta}^{\der}
	 =\MT(\fX_{\eta_\CC})_{\QQ_\ell}^{\der}
	 \subseteq\rho^{\tw}(\GSpin_\ell).
	\]
		Zarhin's description and $(\MTC_2)$ identify the connected center of $\bG_{\ell,2,\eta}$ with the homothety torus. Hence \Cref{cor:Levi subgroup isogenous to G2} shows that the connected center of $\bL_{\ell,\eta}$ is one-dimensional. It contains the image of the weight cocharacter and is therefore contained in the twisted LLV image. Since a connected reductive group is generated by its connected center and derived subgroup, $\bL_{\ell,\eta}\subseteq\rho^{\tw}(\GSpin_\ell)$.

	The generic unipotent radical is $\bP_{\ell,\eta}^{\circ}$, which centralizes the twisted LLV image by \Cref{lem:defect group commutes with twisted LLV}. All generic Levi subgroups are conjugate by this radical, so $\bL_{\ell,\eta}$ is the unique generic Levi subgroup. Under cospecialization, a special-fiber Levi subgroup $\bL_{\ell,s}$ is a reductive subgroup of $\bG_{\ell,\eta}$. By \Cref{lem:Mostow containment}, it is contained in a generic Levi subgroup, hence in $\bL_{\ell,\eta}$. Finally, \Cref{lem:LLV representation is deformation invariant} identifies the twisted LLV images at $s$ and $\eta$, proving the assertion.
\end{proof}
\begin{prop}\label{prop:Levi equals MT}
	Let $X/K$ be a hyper-Kähler variety. Then every Levi subgroup $\bL_\ell\subseteq\bG_\ell(X)^\circ$ satisfies
	\[
	 \bL_\ell=\MT(X_\CC)_{\QQ_\ell}
	\]
	inside $\bG_{\mot}(X)_{\QQ_\ell}$.
\end{prop}
\begin{proof}
	Let $q=\pi_{\ell,2}|_{\bR_{\QQ_\ell}}$. On Lie algebras,
	\[
	 dq(A+a h^{\deg})=A+2a\,\id_{\rH^2},
	\]
	so $q$ has finite kernel. By \Cref{prop:cocharacters in Spin group} and \Cref{rmk:MT inside GSpin group}, both $\bL_\ell$ and $\MT(X_\CC)_{\QQ_\ell}$ are connected closed subgroups of $\bR_{\QQ_\ell}$. Moreover,
	\[
	 q(\bL_\ell)=\bG_{\ell,2}(X)
	 =\MT_2(X_\CC)_{\QQ_\ell}
	 =q(\MT(X_\CC)_{\QQ_\ell}),
	\]
	where the middle equality is $(\MTC_2)$ and the outer equalities follow from \Cref{cor:Levi subgroup isogenous to G2,lem:project MT group}.

	Both groups lie in $Q=\bigl(q^{-1}(\bG_{\ell,2}(X))\bigr)^\circ$. Since $q$ is finite,
	\[
	 \dim Q=\dim\bG_{\ell,2}(X)=\dim \bL_\ell
	 =\dim\MT(X_\CC).
	\]
	Thus both connected subgroups equal $Q$.
\end{proof}
\begin{cor}\label{cor:Levisubgroup is unique}
There is a unique Levi subgroup $\bG_{\ell}^{\red}(X)\subseteq\bG_{\ell}(X)^\circ$. Moreover, for every $i$, its image in $\bG_{\ell,i}(X)^\circ$ is the unique degree-$i$ Levi subgroup.
\end{cor}
\begin{proof}
    The total Levi is unique by \Cref{prop:Levi equals MT}.

    The image of the total Levi in degree $i$ is a Levi subgroup. Under the surjection $\bG_\ell(X)^\circ\to\bG_{\ell,i}(X)^\circ$, the image of the unipotent radical is the unipotent radical \cite[Chapter VIII, Theorem 4.4]{Hoc81}. It centralizes the image of the total Levi by \Cref{lem:defect group commutes with twisted LLV}. Since all degree-$i$ Levi subgroups are conjugate by that radical, the image is the unique degree-$i$ Levi subgroup.
\end{proof}
Henceforth write
\[
\bG_{\ell,i}^{\red}(X)\coloneqq
\pi_{\ell,i}\bigl(\bG_\ell^{\red}(X)\bigr),
\]
and use $\bG_{\ell,+}^{\red}(X)$ analogously for the even image.
\subsection{The even defect group}\label{subsec:connectivity}
Consider the exact sequence
\begin{equation}\label{eq:ProjExactSSAlgMonodromy}
1\longrightarrow \bP_{\ell,+}\longrightarrow\bG_{\ell,+}(X)
\xlongrightarrow{\pi_{\ell,2}}\bG_{\ell,2}(X)\longrightarrow1.
\end{equation}

\begin{theorem}\label{thm:isogeny of algebraic monodromy groups underprojection}
For every hyper-Kähler variety $X/K$, the group $\bP_{\ell,+}$ is connected. Moreover,
\[
\bG_{\ell,+}^{\red}(X)=\MT_+(X_\CC)_{\QQ_\ell}
\xrightarrow{\sim}\bG_{\ell,2}(X).
\]
In particular, $\bP_{\ell,+}=R_u(\bG_{\ell,+}(X))$.
\end{theorem}
\begin{proof}
By \Cref{thm:rank of ell algebraic monodromy groups}, $\bP_{\ell,+}^\circ=R_u(\bG_{\ell,+}(X))$. Projecting the equality of \Cref{prop:Levi equals MT} to even cohomology gives
\[
\bG_{\ell,+}^{\red}(X)=\MT_+(X_\CC)_{\QQ_\ell}.
\]
By \Cref{lem:project MT group} and $(\MTC_2)$, the degree-two projection restricts to an isomorphism from this Levi subgroup onto $\bG_{\ell,2}(X)$. In the Levi decomposition
\[
\bG_{\ell,+}(X)=\bP_{\ell,+}^\circ\rtimes
\bG_{\ell,+}^{\red}(X),
\]
the kernel of degree-two projection is therefore exactly $\bP_{\ell,+}^\circ$. Thus $\bP_{\ell,+}=\bP_{\ell,+}^\circ$.
\end{proof}

\begin{cor}\label{cor:splitting of second degree projection}
The sequence \eqref{eq:ProjExactSSAlgMonodromy} has a unique splitting
\[
\sigma_\ell\colon\bG_{\ell,2}(X)\longrightarrow\bG_{\ell,+}(X),
\]
and multiplication gives a direct product
\[
\bG_{\ell,+}(X)=
\sigma_\ell\bigl(\bG_{\ell,2}(X)\bigr)\times \bP_{\ell,+}.
\]
When $b_2(X)\ne3$, the splitting is the $\ell$-adic realization of the motivic splitting recalled in \eqref{eq:motivic decomposition}.
\end{cor}
\begin{proof}
Take $\sigma_\ell$ to be the inverse of the degree-two projection on the canonical Levi. Thus
\[
\Image(\sigma_\ell)=\MT_+(X_\CC)_{\QQ_\ell}.
\]
This Levi lies in the twisted LLV image by \Cref{prop:cocharacters in Spin group}, while the defect group centralizes that image by \Cref{lem:defect group commutes with twisted LLV}. Hence the product is direct. Every other splitting has a Levi image; all Levi subgroups are conjugate by the unipotent radical, whose conjugation on the canonical Levi is trivial. Thus the splitting is unique. If $b_2(X)\ne3$, its compatibility with the motivic splitting follows from $\MT_+(X_\CC)\xrightarrow{\sim}\MT_2(X_\CC)$. No motivic splitting is asserted here when $b_2(X)=3$.
\end{proof}

\subsection{Semisimple Mumford--Tate conjecture and semisimplicity}\label{subsec:semisimple Mumford--Tate conjecture and semisimplicity}
Use the canonical Levi notation $\bG_{\ell,i}^{\red}(X)$ introduced in \Cref{cor:Levisubgroup is unique}. By \Cref{lem:Levi realizes semisimplification}, it is the connected monodromy group of a semisimplification of the degree-$i$ representation.

For $0<i<2\dim X$, consider the following two conditions:
\begin{itemize}[align=left,leftmargin=*]
    \item[(SS)] $V_{\ell,i}$ is semisimple as a $G_K$-module;
    \item[$(\SMTC)$] under the comparison isomorphism,
    \[
    \bG_{\ell,i}^{\red}(X)=\MT_i(X_\CC)_{\QQ_\ell}
    \]
    as subgroups of $\GL\bigl(\rH_B^i(X,\QQ)\otimes_\QQ\QQ_\ell\bigr)$.
\end{itemize}
We call $(\SMTC)$ the \emph{semisimple Mumford--Tate conjecture} in degree $i$.

\begin{theorem}\label{thm:semisimple Mumford--Tate conjecture even part}
Let $X$ be a hyper-Kähler variety over a finitely generated extension $K/\QQ$. For every prime $\ell$ and every $0\leq i\leq2\dim X$,
\[
\bG^{\red}_{\ell,i}(X)=\MT_i(X_\CC)_{\QQ_\ell}.
\]
\end{theorem}
\begin{proof}
By \Cref{prop:Levi equals MT}, $\bG_\ell^{\red}(X)=\MT(X_\CC)_{\QQ_\ell}$ inside the motivic Galois group. Projection to degree $i$ gives the asserted equality.
\end{proof}

\begin{cor}\label{lem:semisimple Mumford--Tate conjecture}
For a fixed prime $\ell$, the $\ell$-adic equality in $(\MTC_i)$ holds if and only if $V_{\ell,i}$ is semisimple. Consequently, $(\MTC_i)$ holds if and only if $V_{\ell,i}$ is semisimple for every prime $\ell$.
\end{cor}
\begin{proof}
Zariski density and faithfulness give
\[
V_{\ell,i}\text{ is semisimple}
\quad\Longleftrightarrow\quad
\bG_{\ell,i}(X)^\circ\text{ is reductive}.
\]
In that case $\bG_{\ell,i}(X)^\circ=\bG_{\ell,i}^{\red}(X)$, and the result follows from \Cref{thm:semisimple Mumford--Tate conjecture even part}.
\end{proof}

\begin{remark}
Condition (SS) remains open in general. It is known for Galois representations arising from motives in $\AM_K(\sAb)$; for abelian varieties, see \cite[VI, \S3, Theorem~1(a)]{RationalPoints}. It also follows from the Tate conjecture in the form considered in \cite[Theorem~1]{Moonen19}.
\end{remark}

\begin{lemma}\label{lem:finite of P implies semisimple}
Let $X$ be a hyper-Kähler variety over a finitely generated field $K/\QQ$. Then the semisimplicity condition (SS) holds for $X$ in degree $i$ if and only if the action of $\bP_{\ell}^{\circ}$ on $\rH^{i}_{\et}(X_{\overline{K}},\QQ_{\ell})$ is trivial.
\end{lemma}
\begin{proof}
Consider the surjection $\bG_\ell\twoheadrightarrow\bG_{\ell,i}$. Since $\bP_\ell^\circ=R_u(\bG_\ell^\circ)$ by \Cref{thm:rank of ell algebraic monodromy groups}, its image is $R_u(\bG_{\ell,i}^\circ)$ (\cite[Chapter VIII, Theorem 4.4]{Hoc81}). Thus $\bP_\ell^\circ$ acts trivially on degree $i$ exactly when this unipotent radical is trivial, which is equivalent to semisimplicity.
\end{proof}

\subsection{Mumford--Tate conjecture in families}
\label{subsec:semisimple Mumford--Tate conjecture for HKs}

We first specify the terminology used for fibers. A \emph{finite-type point} of $B$ means a morphism $b\colon\Spec L\to B$ with $L/\QQ$ finitely generated, and its fiber is
\[
 \fX_b\coloneqq\fX\times_B\Spec L.
\]
If $\bar b\in B$ is the scheme-theoretic image of $b$, then $\fX_b$ is the base change of $\fX_{\bar b}$ from $k(\bar b)$ to $L$. The image of $G_L\to G_{k(\bar b)}$ is open: indeed, the relative algebraic closure of $k(\bar b)$ in $L$ is finite over $k(\bar b)$, and the remaining extension is regular. Hence this base change alters neither the connected algebraic monodromy group nor the corresponding Mumford--Tate equality. Since $B$ is of finite type over a finitely generated field, every scheme-theoretic point of $B$ has residue field finitely generated over $\QQ$. Thus this terminology includes both ordinary scheme-theoretic fibers and the field-valued fibers used in the definition of deformation equivalence.

We now apply Cadoret's results \cite{Cad15} to obtain deformation invariance of semisimplicity.
\begin{theorem}\label{thm:reductivity in family}
		Let $f\colon\fX\to B$ be a smooth projective family of hyper-Kähler varieties over a smooth geometrically connected variety over a finitely generated characteristic-zero field. Fix a prime $\ell$ and a degree $0\leq i\leq2\dim(\fX/B)$. If $\bG_{\ell,i}(\fX_{b_0})^\circ$ is reductive for one finite-type point $b_0$ of $B$, then $\bG_{\ell,i}(\fX_b)^\circ$ is reductive for every finite-type point $b$ of $B$.
\end{theorem}
\begin{proof}
	By the preceding observation, we may replace $b_0$ and $b$ by their scheme-theoretic images. Apply \cite[Theorems~1.2(3) and~3.5]{Cad15} to the geometric motivic local system $\VV_{\ell,i}=R^i f_*\QQ_\ell$, using the reductive fiber $b_0$. The generic point lies outside the exceptional locus, so its degree-$i$ monodromy group is reductive. By \Cref{lem:finite of P implies semisimple}, the generic defect radical $\bP_{\ell,\eta}^{\circ}$ acts trivially on $\rH^i(\fX_{\bar\eta},\QQ_\ell)$.

	For a specialization path $\eta\rightsquigarrow b$, cospecialization embeds the connected total monodromy group at $b$ into the generic one, compatibly with projection to degree two. Taking kernels gives $\bP_{\ell,b}^{\circ}\hookrightarrow \bP_{\ell,\eta}^{\circ}$. Hence $\bP_{\ell,b}^{\circ}$ also acts trivially on degree $i$. Applying \Cref{lem:finite of P implies semisimple} once more proves that $\bG_{\ell,i}(\fX_b)^{\circ}$ is reductive.
\end{proof}

By \Cref{thm:semisimple Mumford--Tate conjecture even part}, the Mumford--Tate conjecture in degree $i$ for a fiber is equivalent to reductivity of its degree-$i$ monodromy group. The latter is deformation invariant by \Cref{thm:reductivity in family}, giving the following corollary.
\begin{cor}\label{cor:MTC in families}
		Let $f \colon \fX \to B$ be a smooth projective family of hyper-Kähler varieties over a smooth geometrically connected variety $B/K$, where $K/\QQ$ is finitely generated. Fix $0\leq i\leq2\dim(\fX/B)$. If $(\MTC_i)$ holds for one finite-type fiber $\fX_{b_0}$, then $(\MTC_i)$ holds for every finite-type fiber $\fX_b$.
\end{cor}
\begin{proof}
	Fix a prime $\ell$. By \Cref{thm:semisimple Mumford--Tate conjecture even part}, it is sufficient to prove that $\bG_{\ell,i}(\fX_b)^{\circ}$ is reductive. This holds at $b_0$ by the $\ell$-adic equality in $(\MTC_i)$, and hence at every $b$ by \Cref{thm:reductivity in family}. Since this argument applies to every prime $\ell$, $(\MTC_i)$ holds for every finite-type fiber.
\end{proof}

\subsection{Semisimplicity for known examples}\label{subsec:furtherremarks}

For hyper-Kähler varieties belonging to the four established deformation types, semisimplicity can be deduced directly from their cohomological structure, without using the full motivic information.

Using \Cref{thm:reductivity in family}, we may therefore restrict our attention to a concrete construction representing each of these deformation types.

\subsubsection{Hilbert schemes of points}
Let $S$ be a smooth projective surface over a finitely generated field $K$. By classical results of Nakajima \cite{Nak97} and Grojnowski \cite{Gro96} (see also \cite[Chapter~8]{Nak99}), the direct sum of the $\ell$-adic étale cohomology groups
\[
\bigoplus_{n \geq 0} \rH^{4n- \bullet}_{\et}(S^{[n]}_{\overline{K}},\QQ_{\ell})
\]
is an irreducible representation of the Heisenberg superalgebra generated by the shifted (co)homology
\[
\rH^{4-\bullet}_{\et}(S_{\overline K},\QQ_\ell)
\]
of $S$. Moreover, the highest weight vector is the class $[\Spec(K)]$ in $\rH^{0}_{\et}(S_{\overline K}^{[0]},\QQ_\ell)$. The action of the Heisenberg superalgebra is given by some classes
\[
P_\alpha[i] = \varpi_*\left(\Pi^*\alpha \cap [P[i]]\right)\in \rH_{\et}((S^{[n-i]} \times S^{[n]})_{\overline{K}},\QQ_{\ell}),
\]
where $\alpha \in \rH^{\bullet}_{\et}(S_{\overline{K}},\QQ_{\ell})$, $i \in \ZZ$, $[P[i]]$ is an algebraic $(2n-i+1)$-cycle on $S^{[n-i]} \times S^{[n]} \times S$, and
\[
\begin{tikzcd}[row sep= small, column sep = small]
        & S^{[n-i]} \times S^{[n]} \times S \ar[ld,"\varpi"] \ar[rd,"\Pi"'] & \\
      S^{[n-i]} \times S^{[n]} & & S
\end{tikzcd}
\]
are the respective projections. From this, we can deduce that, for any $0 \leq k \leq 4n$, \[
\rH_{\et}^k(S_{\overline{K}}^{[n]},\QQ_{\ell}) \in \langle \rH^{\bullet}_{\et}(S_{\overline{K}},\QQ_{\ell}) \rangle^{\otimes}
\]
as a $G_K$-representation after a finite field extension.

If $S$ is a K3 or abelian surface, then the $G_K$-representation \(\rH^{\bullet}_{\et}(S_{\overline{K}},\QQ_{\ell})\) is semisimple by \cite{Del72}. It follows that $\rH_{\et}^k(S_{\overline{K}}^{[n]},\QQ_{\ell})$ is semisimple for all $n \geq 1$.

\begin{remark}
    In fact, motivic decompositions for Hilbert schemes of points on a smooth algebraic surface \cite{dMM} show that the motive $\fh(S^{[n]})$ is generated from $\fh(S)$ by taking direct sums and subquotients. Thus the semisimplicity of the Galois representations of $S^{[n]}$ follows from that of $S$. This motivic decomposition was also used in \cite{Sol22} and \cite{FFZ21} to establish the abelianicity of the André motives of K3$^{[n]}$-type varieties and $\Kum_n$-varieties.
\end{remark}

\begin{remark}
   We thank Floccari for pointing out that, for the Hilbert scheme of $n$ points $X=S^{[n]}$ on a K3 surface $S$, the $\fg_0$-equivariant decomposition defined by Markman in \cite[\S 4.2]{Mar08} can be used to establish the semisimplicity of $X$ (see also \cite[Theorem~6.2]{Flo22}). In fact, the full cohomology $\rH^{\bullet}_{\et}(X_{\overline{K}},\QQ_{\ell})$ is generated as a $\QQ_{\ell}$-algebra by a subspace
\[
    C_{\ell}^{\bullet} \coloneqq \bigoplus_{i \geq 1} C^{i}_{\ell},
\]
where $C_{\ell}^{i} \subseteq \rH^{i}_{\et}(X_{\overline{K}},\QQ_{\ell})$ is a $\fg_{0,\ell}$-subrepresentation. For $X=S^{[n]}$, the odd cohomology vanishes, and the irreducible factors in each $C_{\ell}^{2k}$ (as an $\SO(\rH^2)$-module) appear with multiplicity at most one. Arguments similar to those in the proof of \Cref{prop:semisimple if multiplicity one} imply that the Galois action on $C_{\ell}^{\bullet}$ is semisimple. Therefore, $X$ satisfies the semisimplicity conjecture.

\end{remark}

\subsubsection{Generalized Kummer varieties}
If $S=A$ is an abelian surface, semisimplicity for the generalized Kummer variety $K_n(A)$ follows from the motivic decomposition and abelian-motive results used in \cite{FFZ21,Sol22}. Merely observing that $K_n(A)$ is a fiber of the isotrivial summation fibration $A^{[n+1]}\to A$ does not by itself give a Galois-equivariant direct summand and is therefore insufficient as a proof.

\subsubsection{$\OG6$-type varieties}
As in \cite[\S 4.2]{Sol22}, semisimplicity for a hyper-Kähler variety of $\OG6$-type follows from the corresponding result for varieties of K3$^{[3]}$-type, using a dominant generically finite rational map from a K3$^{[3]}$-type variety constructed in \cite{MRS}.

\subsubsection{The LLV multiplicity-one criterion}
We can also use the LLV representation to study the semisimplicity of Galois representations attached to hyper-Kähler varieties, especially those of $\OG10$-type.

Since the LLV algebra $\fg \cong \fso(\widetilde{\rH}(X))$ is semisimple, one may consider the decomposition of the $\fg$-module $\rH^*(X)$ into irreducible $\fg$-modules
\begin{equation}\label{eq:whole-coh-decomp}
	\rH^*(X) \otimes_E \overline{E} = \bigoplus_{\mu \in \Delta^+_r} V_{\mu}^{\oplus m_{\mu}}\,,
\end{equation}
where $r = \lfloor b_2(X)/2 \rfloor$, $\Delta_r^+$ is the set of dominant weights of $\fg$, and $\mu = \sum_{i=0}^r\mu_i \epsilon_i \in \Delta_r^+$ is the highest weight of the irreducible representation $V_{\mu}$. Here $\{\pm \epsilon_i\}$ is the set of weights of the standard representation $\widetilde{\rH}(X)$, and the decomposition \eqref{eq:whole-coh-decomp} is called the \emph{LLV decomposition} of $X$.

\begin{example}
    According to the Weyl construction, $V_{(n)}$ is the ``largest'' irreducible subrepresentation of $\mathrm{Sym}^nV$. Moreover, Verbitsky has shown that, for a hyper-Kähler variety $X$ of dimension $2n$, the subalgebra $\SH(X) \subset \rH^*(X)$ generated by $\rH^2(X)$ is isomorphic to the irreducible $\fg$-module $V_{(n)} \subset \Sym^nV$ with highest weight $\mu = (n)$. The summand $\SH(X) \cong V_{(n)}$ is called the \emph{Verbitsky component} of $\rH(X)$. By construction, $\SH(X) \subseteq \rH^{+}(X)$.

\end{example}
    Since $\mathfrak{g}$ is semisimple, the cohomology admits a $\mathfrak{g}$-module decomposition
\begin{equation}
	\rH(X) = V_{(n)} \oplus V' \,.
\end{equation}
For an arbitrary hyper-Kähler manifold $X$, the multiplicity of $V_{(n)}$ in $\rH^*(X)$ is one; that is, $V'$ does not contain an irreducible $\fg$-module of highest weight $\mu = (n)$. Moreover, every $\mu$ that appears in $V'$ satisfies $|\mu|\le n-1$; see \cite[Proposition~2.34]{GKLR22}. In \locc, the authors compute the LLV decompositions for K3$^{[n]}$-type, $\Kum_n$-type, $\OG6$-type, and $\OG10$-type varieties.
\begin{example}\label{ex:OG10}
	According to \cite[Theorem 3.26]{GKLR22}, the LLV decomposition of $\OG10$ is
\[
\rH(X) = V_{(5)} \oplus V_{(2,2)}\,,
\]
i.e., $V' = V_{(2,2)}$ is irreducible. For $\OG6$-type, we have
\[
V' = V_{(1,1,1)} \oplus V^{\oplus 135} \oplus E^{\oplus 240}
\]
	by \cite[Theorem~3.39]{GKLR22}.
\end{example}

\begin{prop}\label{prop:semisimple if multiplicity one}
	If all irreducible factors in the LLV decomposition \eqref{eq:whole-coh-decomp} of $X$ have multiplicity $\leq 1$, then the Galois representation on $\rH^{\bullet}_{\et}(X_{\overline{K}},\QQ_{\ell})$ is semisimple.
\end{prop}
\begin{proof}
By \Cref{thm:rank of ell algebraic monodromy groups}, $\bP_\ell^\circ=R_u(\bG_\ell(X)^\circ)$. After extending scalars to $\overline\QQ_\ell$, \Cref{lem:defect group commutes with twisted LLV} shows that $\bP_\ell^\circ$ centralizes the LLV algebra. Multiplicity one therefore forces it to preserve each irreducible summand $V_\mu$ individually, and Schur's lemma says that it acts on $V_\mu$ through scalars. A scalar element of a connected unipotent group is the identity. Hence $\bP_\ell^\circ$ acts trivially on every $V_\mu$, and therefore on total cohomology. The defining representation of $\bG_\ell(X)$ is faithful, so $\bP_\ell^\circ=1$. Thus $\bG_\ell(X)^\circ$ is reductive, which is equivalent to semisimplicity of the Galois representation.

\end{proof}

As we have seen in \Cref{ex:OG10}, the factors in the LLV decomposition of an $\OG10$-variety have multiplicity $\leq 1$. Thus we obtain the Mumford--Tate conjecture for $\OG10$-varieties.
\begin{cor}\label{cor:MTC OG10}
Let $X$ be an $\OG10$-variety over $K$. Then the Galois representation $\rH^{\bullet}_{\et}(X_{\overline{K}},\QQ_{\ell})$ is semisimple and the Mumford--Tate conjecture holds for $X$.
\end{cor}

\begin{remark}
    The argument in \Cref{prop:semisimple if multiplicity one} is similar to that in \cite[Proposition~6.1]{Flo22} for the motivic Mumford--Tate conjecture for $\OG10$-varieties. For this reason, \Cref{cor:MTC OG10} does not give an essentially new proof of the Mumford--Tate conjecture for $\OG10$-varieties.
\end{remark}

\section{Applications}\label{sec:applications}
\subsection{Arithmetic Nagai conjecture}\label{subsec:Nagai}

Recall that a hyper-Kähler variety $X$ over a $p$-adic local field $K_v$ is of \emph{Type I reduction} if
\[
\rH^2_{\et}(X_{\overline{K}_v},\QQ_{\ell'})
\]
is potentially unramified for \emph{some} prime $\ell' \neq p$. It is then potentially unramified for every prime different from $p$: for $b_2(X)>3$ this is \cite[Corollary~5.3]{IIKTZ25}, while for $b_2(X)=3$ it follows from \cite[Lemma~5.5]{IIKTZ25}, since the rank-two primitive orthogonal Lie algebra contains no nonzero nilpotent element.

The semisimple Mumford--Tate conjecture enables us to extend \cite[Corollary~4.3]{IIKTZ25}, which was originally established for the four known deformation types, to every projective hyper-Kähler variety. Consequently, we can establish the \emph{arithmetic Nagai conjecture} (Conjectures~1.3 and~1.4 in \locc) in this general framework.

\begin{theorem}\label{thm:Nagai for Type I}
	Let $X$ be a hyper-Kähler variety over a $p$-adic local field $K_v$ with Type I reduction. Then $\rH^{i}_{\et}(X_{\overline{K}_v},\QQ_{\ell})$ is potentially unramified for every $0\leq i\leq4n$ and every prime $\ell\neq p$.
\end{theorem}

First, we note that \Cref{thm:semisimple Mumford--Tate conjecture even part} controls the semisimple part of a Frobenius lift through a Levi subgroup of the local algebraic monodromy group.

\begin{prop}\label{prop:padic monodromy group in MT group}
	Let $X$ be a hyper-Kähler variety over a $p$-adic local field $K_v$. Choose a finitely generated subfield $K\subseteq K_v$, a hyper-Kähler variety $X_0/K$ such that $X_0\times_KK_v\cong X$, and an embedding $K\hookrightarrow\CC$. Then every Levi subgroup $\bL_{\ell,v}$ of the identity component of the local $\ell$-adic monodromy group satisfies
	\[
	 \bL_{\ell,v}\subseteq\MT((X_0)_{\CC})_{\QQ_\ell}.
	\]
\end{prop}
\begin{proof}
	After choosing compatible geometric points, the representation of $G_{K_v}$ is the pullback of the representation of $G_K$. Hence the local algebraic monodromy group is a closed subgroup of the global group $\bG_\ell(X_0)$. Its Levi $\bL_{\ell,v}$ is therefore a reductive subgroup of the global group. By \Cref{lem:Mostow containment}, it is contained in a global Levi subgroup. The global Levi is unique by \Cref{cor:Levisubgroup is unique} and equals $\MT((X_0)_\CC)_{\QQ_\ell}$ by \Cref{prop:Levi equals MT}. This proves the claimed containment. Notice that no natural injection between two independently chosen Levi quotients is being asserted.
\end{proof}
\begin{remark}
	In contrast to the global field case, the $\ell$-adic monodromy group $\bG_{\ell}(X)$ for $\ell\neq p$ is not expected to be reductive over a local field $K_v$. Indeed, the inertia representation $\rho_{\ell}|_{I_v}$ may contribute a normal subgroup through the short exact sequence
	\[
	1 \longrightarrow I_v \longrightarrow G_{K_v} \longrightarrow G_{\FF_q} \longrightarrow 1
	\]
	when the total $\ell$-adic cohomology of $X$ is not potentially unramified; the inertia action is quasi-unipotent by Grothendieck's monodromy theorem.

		For this reason, a Levi subgroup $\bL_{\ell,v}$ is not in general the $\ell$-adic algebraic monodromy group of some Frobenius semisimplification.
\end{remark}

\begin{proof}[Proof of \Cref{thm:Nagai for Type I}]
	Choose a finitely generated subfield $K\subseteq K_v$ and a hyper-Kähler model $X_0/K$ as in the preceding proposition; such a descent exists because $X$ is of finite presentation. Fix $\ell\neq p$. After a finite extension of $K_v$, inertia acts unipotently and the local algebraic monodromy group is connected. Let $N_\ell$ be the monodromy operator on total cohomology. Type I means that its degree-two projection $N_{\ell,2}$ vanishes; hence $N_\ell\in\Lie(\bP_{\ell,0})$, where
	\[
	 \bP_{\ell,0}:=\ker\bigl(\bG_\ell(X_0)\longrightarrow
	 \bG_{\ell,2}(X_0)\bigr)
	\]
	is the global defect group of $X_0$.

	Let $\Phi_v$ be a lift of \emph{geometric} Frobenius, let $q_v$ be the cardinality of the residue field, and put $s_v=\rho_\ell(\Phi_v)^{\semisimple}$. The Zariski closure $D_v=\overline{\langle s_v\rangle}^{\Zar}$ is diagonalizable. Choose $m\geq1$ that annihilates the finite component group $D_v/D_v^\circ$, and put
	\[
	 T_v=\overline{\langle s_v^m\rangle}^{\Zar}=D_v^\circ.
	\]
	Thus $T_v$ is a torus. By \Cref{lem:Mostow containment}, it is contained in a local Levi subgroup, which lies in $\MT((X_0)_\CC)_{\QQ_\ell}$ by \Cref{prop:padic monodromy group in MT group}. The Mumford--Tate factor centralizes the global defect group by \Cref{lem:defect group commutes with twisted LLV}; consequently
	\[
	 \Ad(s_v^m)(N_\ell)=N_\ell.
	\]
	On the other hand, the Weil--Deligne relation for geometric Frobenius is
	\[
	 \Ad(s_v^m)(N_\ell)=q_v^{-m}N_\ell.
	\]
	Since $q_v^m\neq1$, one has $N_\ell=0$. Grothendieck's monodromy theorem then shows that inertia has finite image in every degree, i.e., all the representations are potentially unramified.
\end{proof}

\subsection{Maximality of Galois action}

In this final subsection, we establish \Cref{conj:maximality of Galois action} for hyper-Kähler varieties, using the notation of \S\ref{sss:Maximality}.
\begin{theorem}\label{thm:maximality of Galois action}
	Let $X$ be a hyper-Kähler variety over a finitely generated field $K$. Then the $\ell$-adic Lie group $\Gamma_{\ell}^{\simplyconnected}$ is a hyperspecial maximal compact subgroup of $\bG_{\ell}^{\simplyconnected}(\QQ_{\ell})$ for all sufficiently large primes $\ell$.
\end{theorem}
\begin{proof}
	If $b_2(X)=3$, then $\bG_{\ell,2}(X)^\circ$ is a torus by \Cref{thm:Mumford--Tate conjecture HK degree 2}. The isogeny of \Cref{cor:Levi subgroup isogenous to G2} therefore makes the semisimple quotient of the total connected monodromy group trivial. Both $\bG_\ell^{\simplyconnected}$ and $\Gamma_\ell^{\simplyconnected}$ are trivial, so the assertion is automatic. Assume henceforth that $b_2(X)>3$.

	Let $K'/K$ be the uniform finite extension corresponding to the preimage of $\bG_\ell(X)^\circ$ in $G_K$; it is independent of $\ell$ by \Cref{thm:independence of connected components}. By definition, $\Gamma_\ell^{\simplyconnected}$ is obtained from the image of $G_{K'}$ in the semisimple quotient.

	Use the reductive quotient and the groups $\bG_\ell^{\mathrm{ss}}$ and $\Gamma_\ell^{\simplyconnected}$ defined in \S\ref{sss:Maximality}. The isogeny from a compatible total Levi to $\bG_{\ell,2}$ in \Cref{cor:Levi subgroup isogenous to G2} induces an isomorphism of simply connected semisimple groups
	\[
	\pi_{\ell,2}^{\simplyconnected} \colon \bG_{\ell}^{\simplyconnected} \isomto \bG_{\ell,2}^{\simplyconnected}\,.
	\]
	Projecting the representation of the same group $G_{K'}$ gives $\Gamma_{\ell}^{\simplyconnected}= (\pi_{\ell,2}^{\simplyconnected})^{-1} (\Gamma_{\ell,2}^{\simplyconnected}(X_{K'}))$. By \cite[Theorem~1.3(b)]{HL20}, the latter degree-two group for $X_{K'}$ is hyperspecial maximal compact for all sufficiently large $\ell$. Transport through the displayed isomorphism proves the result.
\end{proof}

\bibliographystyle{amsalpha}
\bibliography{MTHK}

\end{document}